\documentclass[12pt]{amsart}
\usepackage{amsmath}
\usepackage{geometry,amsthm,graphics,tabularx,amssymb,shapepar,eucal,verbatim}
\usepackage{latexsym,amsfonts,amssymb,amsmath,amsxtra,bbm,color,mathrsfs}
\usepackage{amscd}
\usepackage{graphicx}
\usepackage[colorlinks=true, allcolors=blue]{hyperref}
\usepackage{pdfsync}
\usepackage[all]{xypic}

\usepackage{braket}
\newcommand{\BC}{{\mathbb {C}}}

\newcommand{\CP}{{\mathcal {P}}}

\newcommand{\RH}{{\mathrm {H}}}

\newcommand{\RM}{{\mathrm {M}}}

\newcommand{\Ad}{{\mathrm{Ad}}}

\newcommand{\Gal}{{\mathrm{Gal}}}
\newcommand{\GL}{{\mathrm{GL}}}

\newcommand{\Hom}{{\mathrm{Hom}}}

\newcommand{\Ind}{{\mathrm{Ind}}}

\newcommand{\N}{{\mathrm{N}}}

\newcommand{\Res}{{\mathrm{Res}}}

\newcommand{\SL}{{\mathrm{SL}}}

\newcommand{\Span}{{\mathrm{Span}}}
\newcommand{\tr}{{\mathrm{tr}}}

\newcommand{\diag}{\operatorname{diag}}

\newcommand{\absk}[1]{\left\lvert#1\right\rvert_k}

\newcommand{\be}{\begin {equation}}
\newcommand{\ee}{\end {equation}}
\newcommand{\bee}{\begin {equation*}}
\newcommand{\eee}{\end {equation*}}

\theoremstyle{Theorem}

\theoremstyle{Theorem}

\theoremstyle{Theorem}

\theoremstyle{Theorem}

\theoremstyle{Plain}

\theoremstyle{remark}

\theoremstyle{remark}

\theoremstyle{Definition}
\newtheorem{dfn}{Definition}[section]

\newtheorem{cord}[dfn]{Corollary}
\newtheorem{prpd}[dfn]{Proposition}
\newtheorem{thmd}[dfn]{Theorem}
\newtheorem{lemd}[dfn]{Lemma}

\newtheorem{remarkd}[dfn]{Remark}

\usepackage{tikz-cd}
\usepackage{extarrows}
\usepackage{bm}
\usepackage{pifont}
\usepackage{float}

\begin{document}

\title[Trilinear and Ginzburg-Rallis models]{On the trilinear and Ginzburg-Rallis models}
\author[X. Wang]{Xinrui Wang}
\address{School of Mathematical Sciences, Zhejiang University, Hangzhou, Zhejiang, P. R. China}
\email{12435014@zju.edu.cn}


\begin{abstract}
    Let $k$ be a non-archimedean local field of characteristic zero. We give sufficient conditions under which the Ginzburg-Rallis models of the induced representations of $\GL_6(k)$  from a parabolic subgroup of type $[2^3]$ are isomorphic to the  trilinear models of the inducing data. We also give nonvanishing criterion for these
    trilinear models and Ginzburg-Rallis models.
\end{abstract}

\maketitle
\tableofcontents

\section{Introduction and the main results}

This paper is devoted to the study of local Ginzburg-Rallis models for $\GL_6$, which is the Whittaker induction of the trilinear models for $\GL_2^3$ in the framework of the 
relative Langlands program (\cite{BZSV}). 

We first recall the definition of the Ginzburg-Rallis model. Let $k$ be a local field of characteristic zero with normalized absolute value $\absk{\cdot}$, and let $ G := \GL_6(k) $. For a positive integer $n$, let $\RM_n(k)$ be the space of $n\times n$ matrices over $k$. The Ginzburg-Rallis subgroup of $G$ is
\[
    H := \left\{
    \begin{bmatrix}
        g & gX & gZ\\
          & g  & gY\\
          &    & g
    \end{bmatrix}
    \, \middle|\, g \in \GL_2(k), X, Y, Z \in \RM_2(k) \right\} = H_0 \ltimes N,
\]
where $ H_0 := \{ \diag(g, g, g) \mid g \in \GL_2(k) \} $ is the image of the diagonal embedding of $\GL_2(k)$ into $G$,  and 
\[
    N := \left\{
    \begin{bmatrix}
        1_2 & X   & Z  \\
            & 1_2 & Y  \\
            &     & 1_2
    \end{bmatrix}
    \, \middle| \, X, Y, Z \in \RM_2(k) \right\}.
\]
Here and henceforth, $1_n$ denotes the $ n \times n $ identity matrix.

Fix a character $ \eta: k^\times \to \mathbb{C}^\times $ and a nontrivial additive character $ \psi: k \to \mathbb{C} $. Then we have the following character 
\[
    \varphi_H: H\to\BC^\times,\quad h =
        \begin{bmatrix}
            g & gX & gZ\\
              & g  & gY\\
              &    & g
        \end{bmatrix}
        \mapsto \eta_H(h) \psi_H(h): = \eta(\det g) \psi(\tr\,X + \tr\,Y).
\]

Let $\pi$ be an admissible smooth representation of $G$, which is of Casselman-Wallach type if $k$ is archimedean.  The Ginzburg-Rallis model of $\pi$ is the space
\[
    \Hom_H \left( \pi, \varphi_H \right),
\]
which is at most one-dimensional if $\pi$ is irreducible by \cite{JSZ, JLX}. In the case that $\pi$ is tempered, it is studied in \cite{W} using local trace formula.

From now on we assume that $k$ is non-archimedean. In this paper we will study the  Ginzburg-Rallis models of certain induced representations of $G$. Let $P$ be the following block upper triangular 
parabolic subgroup of $G$ of type $[2^3]$:
\[
    P = \left\{
    \begin{bmatrix}
        g_1 & X   & Z  \\
            & g_2 & Y  \\
            &     & g_3
    \end{bmatrix}
    \, \middle| \, g_1, g_2, g_3 \in \GL_2(k), X, Y, Z \in \RM_2(k) \right\} = MN,
\]
where $ M = \{ \diag(g_1, g_2, g_3) \mid g_i \in \GL_2(k), i = 1, 2, 3 \} $.

For a topological group $A$, let $\delta_A$ denote its modular character. One computes that for $ g_1, g_2, g_3 \in \GL_2(k) $ and $ X, Y, Z \in \mathrm{M}_2(k) $,
\[
    \delta_P \left(
    \begin{bmatrix}
        g_1 & X   & Z  \\
            & g_2 & Y  \\
            &     & g_3
    \end{bmatrix}
    \right) = \absk{ \det \left( g_1 g_3^{-1} \right) }^4.
\]

Let $ \tau_i, 1 \le i \le 3 $ be irreducible admissible smooth representations of $\GL_2(k)$, and $ \tau := \tau_1 \otimes \tau_2 \otimes \tau_3 $ 
which is a representation of $M$. Extend $\tau$ to a representation of $P$ trivially across $N$ and also denote it by $\tau$. Let $V_\tau$ be the representation space of $\tau$. Consider the induced representation
\[
    \pi_\tau := \Ind^G_P( \tau ) \quad \text{(normalized smooth induction)},
\]
consisting of all locally constant functions $ f: G \rightarrow V_\tau $ satisfying
\[
    f(pg) = \delta_P^{1/2}(p) \tau(p) f(g),\quad \text{for all } p \in P, g \in G.
\]
We will study the Ginzburg-Rallis model $\Hom_H\left(\pi_\tau,\varphi_H\right)$ and compare it with the following trilinear model (see \cite{Prasad})
\[
    \Hom_{H_0} ( \tau, \eta_H ).
\]
Our results will be complementary to the works mentioned above.

Let us recall the classification of irreducible admissible smooth representations of $ \mathrm{GL}_2(k) $ (see \cite[9.11]{B-H}). For a positive integer $n$, denote by $B_n$ the upper triangular Borel subgroup of $\GL_n(k)$.  An irreducible admissible smooth representation $\sigma$ of $\GL_2(k)$ is one of the following:
\begin{itemize}
    \item $\sigma$ is supercuspidal.
    \item $ \sigma = \Ind^{ \GL_2(k) }_{B_2}( \chi_1 \otimes \chi_2 ) $ (normalized smooth induction), where $\chi_1$ and $\chi_2$ are characters of $k^\times$ with $ \chi_1 \chi_2^{-1} \ne \absk{ \cdot }^{ \pm1 }$.
    \item $\sigma$ is the unique irreducible quotient of the representation
    \[
        \Ind^{ \GL_2(k) }_{B_2} \left( \chi \absk{ \cdot }^{-1/2} \otimes \chi \absk{ \cdot }^{1/2} \right) = ( \chi \circ \det ) \otimes \Ind^{ \GL_2(k) }_{B_2} \delta_{B_2}^{-1/2}
    \]
    for some character $\chi$ of $k^\times$, which is called a Steinberg representation and denoted by $ \mathrm{St}( \chi ) $.
    \item $ \sigma = \chi \circ \det $ is the one-dimensional subrepresentation of the representation $ \Ind^{ \GL_2(k) }_{B_2} \left( \chi \absk{ \cdot }^{-1/2} \otimes \chi \absk{ \cdot }^{1/2} \right) $ for some character $\chi$ of $k^\times$.
\end{itemize}
Here, $ \delta_{B_2} = \absk{ \cdot } \otimes \absk{ \cdot }^{-1} $ is the modular character of $B_2$. Infinitely dimensional irreducible smooth representations of $\GL_2(k)$ that are isomorphic to the representations of the form $ \Ind^{ \GL_2(k) }_{B_2}( \chi_1 \otimes \chi_2 ) $ are called principal series representations and others are essential discrete series representations.   

Let $ \omega_{ \tau_i } $ denote the central character of the representation $\tau_i$ for $ i = 1, 2, 3 $. For a smooth representation $V$ of finite length of a totally disconnected group, we denote by $ V_\mathrm{ss} $ the semi-simplification of $V$, i.e., if $V$ admits a filtration $ 0 = V_0 \subseteq V_1 \subseteq \cdots \subseteq V_n = V $ with each $ V_{i + 1} / V_i $ irreducible, then $ V_\mathrm{ss} := \bigoplus_{i = 0}^{n - 1} V_{i + 1} / V_i $. Let 
\[ 
    T_2 = \{ \diag( x_1, x_2 ) \mid x_1, x_2 \in k^\times \} \text{ and } N_2 = \left\{
    \begin{bmatrix}
        1 & x\\
          & 1
    \end{bmatrix}
    \, \middle| \, x \in k \right\}. 
\]
For each smooth representation $V$ of $\GL_2(k)$, define its Jacquet module by 
\[
    J(V) := V / \Span_{ \mathbb{C} } \{ n \cdot v - v \mid n \in N_2, v \in V \}.
\]
It naturally carries a $T_2$-representation structure. 

\begin{dfn}
    Let $\tau_1$, $\tau_2$ and $\tau_3$ be irreducible admissible smooth representations of $\GL_2(k)$. A representation $ \tau = \tau_1 \otimes \tau_2 \otimes \tau_3 $ of $M$ is called \textbf{good} if for each $ i = 1, 2, 3 $ and $ j, l \in \{ 1, 2, 3 \} \setminus \{ i \} $ with $ j < l $, 
    \[
        \Hom_{T_2} \left( J( \tau_i )_\mathrm{ss}, ( \eta \circ \det ) \otimes \left( \omega_{ \tau_j }^{-1} \otimes \omega_{ \tau_l }^{-1} \right) \right) = 0.
    \]   
\end{dfn}

The first main result of this paper is as follows.

\begin{thmd}
    \label{mainthm1}
    Let $ \tau = \tau_1 \otimes \tau_2 \otimes \tau_3 $ be a good representation of $M$. Then it holds that
        \[
            \Hom_H \left( \pi_\tau, \varphi_H \right) \cong \Hom_{H_0} \left( \tau, \eta_H \right).
        \]
\end{thmd}
    
\begin{remarkd}
    For tempered $\tau$, the statement was proved in \cite{W} without requiring $\tau$ to be good.
\end{remarkd}

We will prove the following criterion for good representations.

\begin{prpd}
    \label{goodcr}
    A representation $ \tau = \tau_1 \otimes \tau_2 \otimes \tau_3 $ of $M$ is good if and only if the following conditions are satisfied: 
    \begin{itemize}
        \item[(G1)] For each $ i = 1, 2, 3 $, if $ \tau_i = \chi_i \circ \det $ for some character $\chi_i$ of $k^\times$, then
        \[
            \chi_i \otimes \chi_i \ne ( \eta \circ \det ) \otimes \left( \omega_{ \tau_j }^{-1} \otimes \omega_{ \tau_l }^{-1} \right)
        \]
        where $ j, l \in \{ 1, 2, 3 \} \setminus \{ i \} $ with $ j < l $;
        \item[(G2)] For each $ i = 1, 2, 3 $, if $ \tau_i = \mathrm{St} \left( \chi_i \right) $ for some character $\chi_i$ of $k^\times$, then
        \[
            \chi_i \otimes \chi_i \ne ( \eta \circ \det ) \otimes \left( \omega_{ \tau_j }^{-1} \otimes \omega_{ \tau_l }^{-1} \right) \otimes \delta_{B_2}^{-1},
        \]
        where $ j, l \in \{ 1, 2, 3 \} \setminus \{ i \} $ with $ j < l $;
        \item[(G3)] For each $ i = 1, 2, 3 $, if $ \tau_i = \mathrm{Ind}_{B_2}^{ \GL_2(k) } ( \chi_{ij} \otimes \chi_{il} ) $ for some characters $ \chi_{ij} $ and $ \chi_{il} $ of $k^\times$, where $ j, l \in \{ 1, 2, 3 \} \setminus \{ i \} $ with $ j < l $, then
        \[
            \chi_{ij} \otimes \chi_{il} \ne  ( \eta \circ \det ) \otimes \left( \omega_{ \tau_j }^{-1} \otimes \omega_{ \tau_l }^{-1} \right) \otimes \delta_{B_2}^{-1/2}
        \]
        and
        \[
            \chi_{ij} \otimes \chi_{il} \ne ( \eta \circ \det ) \otimes \left( \omega_{ \tau_l }^{-1} \otimes \omega_{ \tau_j }^{-1} \right) \otimes \delta_{B_2}^{1/2}.
        \]
    \end{itemize}
\end{prpd}

Our next three theorems give the nonvanishing criterion for the trilinear model of a good representation $\tau$ of $M$, which implies the nonvanishing of the Ginzburg-Rallis model 
of $ \pi_\tau $ in view of Theorem \ref{mainthm1}. The criterion consists of the following three cases:
\begin{itemize}
    \item when not all three representations are the essential discrete series, it is given by Corollary \ref{mainthm2};
    \item when all three representations are essential discrete series and at least one of them is a special representation,  it is given by Corollary \ref{mainthm4};
    \item when all three representations are supercuspidal,  it is given by Theorem \ref{mainthm3}.
\end{itemize}

\begin{cord}
    \label{mainthm2}
    Let $\tau_1$, $\tau_2$, and $\tau_3$ be infinite-dimensional irreducible admissible smooth representations of $\GL_2(k)$ such that the product of their central characters equals $\eta^2$. Suppose that not all of $\tau_1$, $\tau_2$, and $\tau_3$ are essential discrete series representations. If $ \tau = \tau_1 \otimes \tau_2 \otimes \tau_3 $ is good, then
    \[
        \dim \Hom_H \left( \pi_\tau, \varphi_H \right) = \dim \Hom_{H_0} ( \tau, \eta_H ) = 1.
    \]
\end{cord}

\begin{cord}
    \label{mainthm4}
    Assume that the residue characteristic of $k$ is odd. Let $\tau_1$, $\tau_2$, and $\tau_3$ be infinite-dimensional essential discrete series representations of $\GL_2(k)$ such that the product of their central characters equals $\eta^2$. Suppose that not all of $\tau_1$, $\tau_2$, and $\tau_3$ are supercuspidal representations. Assume that $ \tau = \tau_1 \otimes \tau_2 \otimes \tau_3 $ is good. If at least one of the conditions are satisfies:
    \begin{itemize}
        \item There is exactly one Steinberg representation, say $ \tau_1 = \mathrm{St}( \chi_1 ) $ for some character $\chi_1$ of $k^\times$. $\tau_2$ and $\tau_3$ are supercuspidal satisfying
        \[
            ( \chi_1 \eta^{-1} \circ \det ) \otimes \tau_2 \ncong \tau_3^\vee.
        \]
        \item There are exactly two Steinberg representations among $\tau_1$, $\tau_2$, and $\tau_3$. 
    \end{itemize}
    Then
    \[
        \dim \Hom_H \left( \pi_\tau, \varphi_H \right) = \dim \Hom_{H_0} ( \tau, \eta_H ) = 1.
    \]
\end{cord}

\begin{thmd}
    \label{mainthm3}
    Assume that the residue characteristic of $k$ is odd. Let $\tau_1$, $\tau_2$, and $\tau_3$ be  irreducible supercuspidal representations of $\GL_2(k)$ such that the product of their central characters equals $\eta^2$. Suppose that for each $i=1,2,3$, the representation $\tau_i$ corresponds to $ \mathrm{Ind}^k_{F_i} \xi_i $ under the Langlands correspondence, where $F_i$ is a separable quadratic extension of $k$ and $\xi_i$ is a character of $F_i^\times$. If any one of the conditions shown in Proposition $ \ref{cr1} $, Proposition $\ref{cr2}$, Proposition $\ref{cr3}$ or Proposition $\ref{cr4}$ holds, then
    \[
        \dim \Hom_H \left( \pi_\tau, \varphi_H \right) = \dim \Hom_{H_0} ( \tau, \eta_H ) = 1.
    \]
\end{thmd}
    
\begin{remarkd}
    Propositions $\ref{cr1}$ and $\ref{cr3}$ are due to Prasad \cite{Prasad}, and Propositions $\ref{cr2}$ and $\ref{cr4}$ are our works in this paper. The proofs of Propositions $\ref{cr2}$ and $\ref{cr4}$ are inspired by the framework in Proposition $\ref{cr1}$. However, while Prasad converts the problem to trilinear forms on $D_k$ via Tunnell's theorem (see \cite[Theorem 8.3]{Prasad} and \cite{Tunnell}), we opt for a direct evaluation of the $\varepsilon$-factors with more fundamental results \cite[Theorems 1.2 and 1.4]{Tunnell}.
\end{remarkd}

\section{Proof of Theorem \ref{mainthm1}}

\subsection{$H$-orbits and stabilizers}

The elements of symmetric group $ \mathfrak{S}_6 $ can be realized as matrices in $\GL_6$. For each permutation $ s \in \mathfrak{S}_6 $, we associate the matrix $ \sum_{i = 1}^6 E_{ s(i), i } $, which we also denote by $s$. Here $ E_{ij} \in \mathrm{M}_6(k) $ is the matrix whose $(i, j)$-th entry is $1$ and all others are $0$. A straightforward computation yields that
\[
    s E_{ij} s^{-1} = E_{ s(i), s(j) }, \quad \text{ for } 1 \le i, j \le 6.
\]

Let $ \mathcal{P} = P \backslash G $. We first describe the orbits of the right action of $H$ on $\mathcal{P}$, equivalently, the double cosets $ P \backslash G/H $. Let $ W \cong \mathfrak{S}_6 $ be the Weyl group of $G$ and let $W_P$ be the Weyl subgroup of $P$, which is generated by the transpositions $(12)$, $(34)$, and $(56)$. 
Consider the collection $W^{P, P}$ of permutations $ s =
\begin{pmatrix}
    1   & 2   & 3   & 4   & 5   & 6  \\
    i_1 & i_2 & i_3 & i_4 & i_5 & i_6
\end{pmatrix} $
satisfying the following conditions:
\begin{itemize}
    \item[(a)]  $ i_1 < i_2 $, $ i_3 < i_4 $, and $ i_5 < i_6 $;
    \item[(b)]  In the sequence $ ( i_1, i_2, \dots, i_6 ) $, the number $1$ appears before $2$, $3$ appears before $4$, and $5$ appears before $6$. 
\end{itemize}
In fact, $W^{P, P}$ is a set of representatives of the double cosets $ W_P \backslash W / W_P $. (See \cite[Lemma 2.11]{B-Z})

\begin{lemd}
    \label{orbit2}
    There is a decomposition
    \[
        G = \bigcup_{ s \in W^{P, P} } \bigcup_{ t \in \Omega } PstH,
    \]
    where
    \[
        \Omega = \left\{ 1_6,
        \begin{bmatrix}
            1_2 &   &   \\
                &w_2&   \\
                &   &1_2
        \end{bmatrix}
        ,
        \begin{bmatrix}
            1_2 &     &    \\
                & 1_2 &    \\
                &     & w_2
        \end{bmatrix}
        ,
        \begin{bmatrix}
            1_2 &     &    \\
                & w_2 &    \\
                &     & w_2
        \end{bmatrix}
        ,
        \begin{bmatrix}
            1_2 &     &       \\
                & w_2 &       \\
                &     &
                \begin{matrix}
                    0 & 1\\
                    1 & 1
                \end{matrix}
        \end{bmatrix}
        \right\}.
    \]
    Here and henceforth for a positive integer $n$, $w_n$ denotes the $n\times n$ anti-diagonal permutation matrix. 
\end{lemd}
    
\begin{proof}
    Set
    \[
        B := \left\{ \diag( b_1, b_2, b_3 ) \, \middle| \, b_1, b_2, b_3 \in B_2 \right\} \cong B_2 \times B_2 \times B_2.
    \]
    By Bruhat decomposition of $\GL_6$ and parabolic subgroup $P$, we obtain
    \[
        G = \bigcup_{ s \in W^{P, P} } PsP.
    \]
    Noting that $ sB \subseteq Ps $ for all $ s \in W^{P, P} $, it suffices to show the decomposition
    \[
        P = \bigcup_{ t \in \Omega } BtH.
    \]
    Set
    \[
        T_2^{(i)} = \{ \diag( x_1, x_2, x_3 ) \mid x_i \in T_2 \text{ and } x_j = 1_2 \text{ for } j \ne i \} \subseteq G.
    \]
    We define analogously the subgroups $ X^{(i)} \subseteq G $ for each subgroup $X$ of $\GL_2(k)$ and $ i = 1, 2, 3 $. One checks easily that 
    \begin{equation}
        \label{XH0}
        X^{(i)} X^{(j)} H_0 = X^{(i)} X^{(l)} H_0
    \end{equation}
    for $ \{ i, j, l \} = \{ 1, 2, 3 \} $ and any subgroup $ X \subseteq \GL_2(k) $. Let $ \mathfrak{S}_2 = \{ 1_2, w_2 \} $ and $ \mathfrak{S}_2' = \mathfrak{S}_2 \cup \left\{
    \begin{bmatrix}
        0 & 1\\
        1 & 1
    \end{bmatrix}
    \right\} $. Then we have the following decomposition:
    \[
        \GL_2(k) = \bigcup_{ x \in \mathfrak{S}_2 } B_2 x B_2 = \bigcup_{ y \in \mathfrak{S}_2' } B_2 y T_2.
    \]
    Hence,
    \[
        \begin{aligned}
            M &= \GL_2(k)^{(2)} \GL_2(k)^{(3)} H_0                              \\
              &= \bigcup_{ x \in \mathfrak{S}_2 } \GL_2(k)^{(3)} B_2^{(2)} \diag( 1_2, x, 1_2 ) B_2^{(2)} H_0                                                     \\
              &= \bigcup_{ x \in \mathfrak{S}_2 } B_2^{(1)} B_2^{(2)} \GL_2(k)^{(3)} \diag( 1_2, x, 1_2 ) H_0                   & ( \text{by } ( \ref{XH0} ) )\\
              &= \bigcup_{ x \in \mathfrak{S}_2 } \bigcup_{ y \in \mathfrak{S}_2' } B_2^{(1)} B_2^{(2)} B_2^{(3) } \diag( 1_2, x, y ) T_2^{(3)} H_0             \\
              &= \bigcup_{ x \in \mathfrak{S}_2 } \bigcup_{ y \in \mathfrak{S}_2' } B \diag( 1_2, x, y ) T_2^{(1)} T_2^{(2)} H_0 & ( \text{by } ( \ref{XH0} ) )\\
              &= \bigcup_{ x \in \mathfrak{S}_2 } \bigcup_{ y \in \mathfrak{S}_2' } B \diag( 1_2, x, y ) H_0.
            \end{aligned}
    \]
    The last equality holds because $ T_2 x = x T_2 $ for each $ x \in \mathfrak{S}_2 $. 
    
    Now observe that
    \[
        \begin{bmatrix}
            1_2 &   &     \\
                &1_2&     \\
                &   &
            \begin{matrix}
                0 & 1\\
                1 & 1
            \end{matrix}
        \end{bmatrix}
        =
        \begin{bmatrix}
            1 & -1 &   &    &   &  \\
              & 1  &   &    &   &  \\
              &    & 1 & -1 &   &  \\
              &    &   & 1  &   &  \\
              &    &   &    & 1 &  \\
              &    &   &    &   & 1
        \end{bmatrix}
        \begin{bmatrix}
            1_2 &     &    \\
                & 1_2 &    \\
                &     & w_2
        \end{bmatrix}
        \begin{bmatrix}
            1 & 1 &   &   &   &  \\
              & 1 &   &   &   &  \\
              &   & 1 & 1 &   &  \\
              &   &   & 1 &   &  \\
              &   &   &   & 1 & 1\\
              &   &   &   &   & 1
        \end{bmatrix}
        ,
    \]
    which implies
    \[
        M = \bigcup_{ x \in \mathfrak{S}_2 } \bigcup_{ y \in \mathfrak{S}_2' } B \diag( 1_2, x, y ) H_0 = \bigcup_{ t \in \Omega } BtH_0.
    \]
    Since $ P = MN $ and $ H = H_0N $, it follows that $ P = \bigcup_{ t \in \Omega } BtH $.
\end{proof}

In the rest of this section, we will denote the matrices
\[
    1_6,
    \begin{bmatrix}
        1_2 &     &    \\
            & w_2 &    \\
            &     & 1_2
    \end{bmatrix}
    ,
    \begin{bmatrix}
        1_2 &     &    \\
            & 1_2 &    \\
            &     & w_2
    \end{bmatrix}
    ,
    \begin{bmatrix}
        1_2 &     &    \\
            & w_2 &    \\
            &     & w_2
    \end{bmatrix}
    ,
    \begin{bmatrix}
        1_2 &     &   \\
            & w_2 &   \\
            &     &
        \begin{matrix}
            0 & 1\\
            1 & 1
        \end{matrix}
    \end{bmatrix}
\]
by $ t_1, t_2, t_3, t_4, t_0 $, respectively. The unique open orbit of the right action of $H$ on $\mathcal{P}$ is represented by $ \gamma_0 := (15) (26) $. 

Let $\Gamma$ denote the set $ \{ st \mid s \in W^{P, P}, t \in \Omega \} \subseteq G $. For any $ \gamma \in \Gamma $, the stabilizer of $ P \gamma \in \mathcal{P} $ in $H$ is
\[
    H_{ \gamma } = \gamma^{-1} P \gamma \cap H.
\]
Set $ P_\gamma = \gamma H_\gamma \gamma^{-1} = P \cap \gamma H \gamma^{-1} $ and $ N_\gamma = \gamma^{-1} N \gamma \cap H \subseteq H_\gamma $. 

\begin{lemd}
    \label{truestabilizer}
    If $ \gamma \in \Gamma $ such that $ { \psi_H }|_{ N_\gamma } $ is trivial, then $\gamma$ is contained in  the orbit of one of the representatives in the set
    \[
        \Gamma' := \{ (15)(26), (152643), (134625), (13)(25)(46) \}.
    \]
\end{lemd}

\begin{proof}
    Recall that the root subgroups of $G$ with respect to the Borel subgroup $B_6$ are defined as follows. For each root $ \varepsilon_i - \varepsilon_j, 1 \le i, j \le 6 $, the corresponding root subgroup is
    \[
        U_{ij} = \{ 1_6 + xE_{i,j} \mid x \in k \} \subseteq G.
    \]
    Assume that $ \psi_{H}|_{ N_\gamma } $ is trivial. Then the root subgroups $ U_{13}, U_{24}, U_{35} $, and $U_{46}$ are not contained in the group $N_\gamma$. 
    
    We begin with the case where $ \gamma = s t_i $ for some $ s \in W^{P, P} $ and $ i = 1, \dots, 4 $. Note that in this case, $\gamma$ is a matrix realized as a permutation in $ \mathfrak{S}_6 $. Therefore, we have
    \[
        U_{ij} \nsubseteq N_\gamma \Longleftrightarrow U_{ \gamma(i), \gamma(j) } = \gamma U_{ij} \gamma^{-1} \nsubseteq N.
    \]
    This means that for each $ i = 1, \dots, 4 $,
    \[
        ( \gamma(i), \gamma(i + 2) ) \in \{ (1,2), (3,4), (5,6) \} \cup \{ (j,l) \mid 1 \le  l < j \le 6 \}.
    \]
    We obtain the following table of $\gamma$ meeting these requirements. 
        \begin{table}[H]
            \centering
            \caption{Representatives of the Filtered Orbits}
            \label{tab:example}
            \begin{tabular}{|c|c|c|c|}
                \hline
                $(15)(26)$     & $(152643)$     & $(134625)$     & $(13)(25)(46)$ \\ 
                \hline
                $(15)(26) t_2$ &                &                &                \\
                \hline
                $(15)(26) t_3$ &                & $(134625) t_3$ &                \\
                \hline
                $(15)(26) t_4$ & $(152643) t_4$ &                &                \\
                \hline
            \end{tabular}
        \end{table}
    One checks directly that the elements in the same column of the preceding table belong to the same orbit. Hence, the four elements $(15)(26)$, $(152643)$, $(134625)$, and $(13)(25)(46)$ are representatives for the four distinct orbits appearing in the table.
    
    We now consider the case where $ \gamma = s t_0 $ with $ s \in W^{P, P} $. The conjugation action of $t_0$ on the root subgroups is given by
    \[
        t_0 U_{ij} t_0^{-1} =
        \begin{cases}
            U_{14},                                                          & \text{if } (i,j) = (1,3),\\
            U_{23},                                                          & \text{if } (i,j) = (2,4),\\
            U_{46-45} := \{ 1_6 + x ( E_{4, 6} - E_{4, 5} ) \mid x \in k \}, & \text{if } (i,j) = (3,5),\\
            U_{35},                                                          & \text{if } (i,j) = (4,6).
        \end{cases}
    \]
    Since the conjugation by $s$ permutes the entries of a matrix, it follows that if $ s U_{46-45} s^{-1} \nsubseteq N $, then both root subgroups $ s U_{45} s^{-1} $ and $ s U_{46} s^{-1} $ are not contained in $N$ either. We therefore obtain that for $ (i,j) = (1,4), (2,3), (3,5), (4,5), (4,6) $,
    \[
        ( s(i), s(j) ) \in \{ (1,2), (3,4), (5,6) \} \cup \{ (l, m) \mid 1 \le m < l \le 6 \}.
    \]
    A direct computation shows that the only element satisfying these requirements is
    \[
        \gamma = (15)(26) t_0,
    \]
    which lies in the same orbit as $(15)(26)$. 
\end{proof}

\begin{prpd}
    \label{groupsforgamma}
    The groups $H_\gamma$, $P_\gamma$, and $N_\gamma$ for
    \[
        \gamma = (15)(26) ( = \gamma_0 ), (152643), (134625), (13)(25)(46)
    \]
    are described as follows. For $ \gamma = \gamma_0 $, the representative of the unique open orbit,
    \[
        H_{ \gamma_0 } = P_{ \gamma_0 } = H_0, N_{ \gamma_0 } = \{ 1_6 \}.
    \]
    We then define the parameterizations $ \mathsf{H}_\gamma, \mathsf{P}_\gamma $ for $\gamma$ representing the non-open orbits in each case.
    \begin{itemize}
        \item[\rm(a)] For $ \gamma = (152643) $:
        \[
            \mathsf{H}_\gamma ( g, h, x, u, v, y ) =
            \begin{bmatrix}
                g & hx &   &    &    &   \\
                  & h  &   &    &    &   \\
                  &    & g & hx & gu & y \\
                  &    &   & h  &    & hv\\
                  &    &   &    & g  & hx\\
                  &    &   &    &    & h
            \end{bmatrix}
        \]
        and
        \[
            \mathsf{P}_\gamma ( g, h, x, u, v, y ) =
            \begin{bmatrix}
                g & gu & hx & y  &   &   \\
                  & g  &    & hx &   &   \\
                  &    & h  & hv &   &   \\
                  &    &    & h  &   &   \\
                  &    &    &    & g & hx\\
                  &    &    &    &   & h
            \end{bmatrix}
            ,
        \]
        where $ g, h \in k^\times $ and $ x, u, v, y \in k $.
        \item[\rm(b)] For $ \gamma = (134625) $:
        \[
            \mathsf{H}_\gamma ( g, h, x, u, v, y ) = 
            \begin{bmatrix}
                g & hx & gu & y  &   &   \\
                  & h  &    & hv &   &   \\
                  &    & g  & hx &   &   \\
                  &    &    & h  &   &   \\
                  &    &    &    & g & hx\\
                  &    &    &    &   & h
            \end{bmatrix}
        \]
        and
        \[
            \mathsf{P}_\gamma (g, h, x, u, v, y ) = 
            \begin{bmatrix}
                g & hx &   &    &    &   \\
                  & h  &   &    &    &   \\
                  &    & g & gu & hx & y \\                        
                  &    &   & g  &    & hx\\ 
                  &    &   &    & h  & hv\\
                  &    &   &    &    & h
            \end{bmatrix}
            ,
        \]
        where $ g, h \in k^\times $ and $ x, u, v, y \in k $.
        \item[\rm(c)] For $ \gamma = (13)(25)(46) $:
        \[
            \mathsf{H}_\gamma ( g, h, x, u, v, y_1, y_2, y_3 ) = 
            \begin{bmatrix}
                g & y_1 &   & y_2 &    & hx \\
                  & h   &   & hv  &    &    \\
                  &     & g & y_1 & gu & y_3\\
                  &     &   & h   &    &    \\
                  &     &   &     & g  & y_1\\
                  &     &   &     &    & h
            \end{bmatrix}
            ,
        \]
        and
        \[
            \mathsf{P}_\gamma ( g, h, x, u, v, y_1, y_2, y_3 ) =
            \begin{bmatrix}
                g & gu &   & y_3 &     & y_1\\
                  & g  &   & y_1 &     &    \\
                  &    & g & hx  & y_1 & y_2\\
                  &    &   & h   &     &    \\
                  &    &   &     & h   & hv \\
                  &    &   &     &     & h
            \end{bmatrix}
            ,
        \]
        where $ g, h \in k^\times $ and $ x, u, v, y_1, y_2, y_3 \in k $.
    \end{itemize}
    Then, in cases {\rm (a)} and {\rm (b)}, we have
    \begin{itemize}
        \item $ H_\gamma = \{ \mathsf{H}_\gamma ( g, h, x, u, v, y ) \mid g, h \in  k^\times, x, u, v, y \in k \} $,
        \item $ P_\gamma = \{ \mathsf{P}_\gamma ( g, h, x, u, v, y ) \mid g, h \in  k^\times, x, u, v, y \in k \} $,
        \item $ N_\gamma = \left\{ \mathsf{H}_\gamma (1, 1, 0, 0, 0, y ) \mid y \in k \right\} $,
    \end{itemize}
    and the parameterizations $ \mathsf{H}_\gamma, \mathsf{P}_\gamma: (k^\times)^2 \times  k^4 \rightarrow G $ satisfy:
    \begin{itemize}
        \item $ \mathsf{P}_\gamma = \gamma \mathsf{H}_\gamma \gamma^{-1} $.
        \item $ \delta_P ( \mathsf{P}_\gamma ( g, h, \cdot ) ) = \absk{ g h^{-1} }^{4} $, and $ \delta_{ H_\gamma } ( \mathsf{H}_\gamma ( g, h, \cdot ) ) = \absk{ g h^{-1} }^2 $ for $ g, h \in k^\times $.
    \end{itemize}
    In case {\rm (c)}, we have 
    \begin{itemize}
        \item $ H_\gamma = \{ \mathsf{H}_\gamma ( g, h, x, u, v, y_1, y_2, y_3 ) \mid g, h \in k^\times, x, u, v, y_1, y_2, y_3 \in k \} $,
        \item $ P_\gamma = \{ \mathsf{P}_\gamma ( g, h, x, u, v, y_1, y_2 ,y_3 ) \mid g, h \in k^\times, x, u, v, y_1, y_2, y_3 \in k \} $,
        \item $ N_\gamma = \{ \mathsf{H}_\gamma ( 1, 1, 0, 0, 0, y_1, y_2, y_3 ) \mid y_1, y_2, y_3 \in k \} $,
    \end{itemize}
    and the parameterizations $ \mathsf{H}_\gamma , \mathsf{P}_\gamma: (k^\times)^2 \times  k^6 \rightarrow G $ satisfy:
    \begin{itemize}
        \item $ \mathsf{P}_\gamma = \gamma \mathsf{H}_\gamma \gamma^{-1} $.
        \item $ \delta_P ( \mathsf{P}_\gamma (g, h, \cdot ) ) = \absk{ g h^{-1} }^8 $, and $ \delta_{ H_\gamma } ( \mathsf{H}_\gamma ( g, h, \cdot ) ) = \absk{ g h^{-1} }^4$ for $ g, h \in k^\times $.
    \end{itemize}
\end{prpd}

\subsection{Reducibility to homology calculations}

In the following two subsections, we prove that for each good representation $ \tau = \tau_1 \otimes \tau_2 \otimes \tau_3 $ and each non-open orbit $ P \gamma H \subseteq P \backslash G $, the homology groups 
\begin{equation}
    \label{targetspaces}
    \RH_i \left( H_\gamma, \tau^\gamma \otimes \rho_P^\gamma|_{ H_\gamma } \otimes \delta^{-1}_{ H_\gamma } \otimes \varphi_H^{-1} \right) = 0,\quad \text{for } i \geq 0.
\end{equation}
Hereafter, for an element $\gamma\in G$, the notation $ ( \cdot )^\gamma $ denotes the representation of $ H_\gamma $ obtained from a representation of $ P_\gamma $ by defining the action of $ h \in H_\gamma $ as that of $ \gamma h \gamma^{-1} \in P_\gamma $. $\rho_P$ denotes the square root of the modular character $\delta_P$.

Before proceeding to the calculations, we briefly explain the motivation and set up the notation used throughout these two subsections. 

Let $\mathsf{E}$ be the vector bundle on $ \mathcal{P} = P \backslash G$ induced by the representation $ \delta_P^{1/2} \otimes \tau $ of $P$. Equivalently,
\[
    \mathsf{E} = G \times_P V_\tau := G \times V_\tau / \sim,
\]
where $V_\tau$ is the representation space of $\tau$ and the equivalence relation is defined by 
\[
    (g, v) \sim ( pg, \delta_P(p)^{1/2} \tau(p) v ), \text{ for } g \in G, p \in P, \text{ and } v \in V_\tau.
\]
We endow $\mathsf{E}$ with the quotient topology induced from the product topology on $ G \times V_\tau $. Define the projection map by
\[
    \mathrm{pr}: \mathsf{E} \rightarrow \mathcal{P}, \quad (g, v) \mapsto Pg.
\]
Let $ \Gamma( \mathcal{P}, \mathsf{E} ) $ denote the space of smooth sections of $\mathsf{E}$:
\[
    \Gamma( \mathcal{P}, \mathsf{E} ) := \left\{ s: \mathcal{P} \rightarrow \mathsf{E} \mid s\text{ is locally constant such that } s \circ \mathrm{pr} = \mathrm{id}_{ \mathcal{P} } \right\},
\]
and let $ \Gamma_c ( \mathcal{P}, \mathsf{E} ) $ denote the subspace of $ \Gamma ( \mathcal{P}, \mathsf{E} ) $ consisting of sections with compact support. Define the action of $G$ on a section $ s \in \Gamma( \mathcal{P}, \mathsf{E} ) $ by right translation. Then as representations of $G$, we have isomorphisms
\begin{equation}
    \label{indandsec}
    \Gamma( \mathcal{P}, \mathsf{E} ) \cong \Ind^G_P ( \tau ) \text{ and } \Gamma_c ( \mathcal{P}, \mathsf{E} ) \cong \mathrm{ind}^G_P ( \tau ).
\end{equation}
Here, $\mathrm{ind}$ denotes the normalized compact induction. By Iwasawa decomposition, we have $ \Ind^G_P ( \tau ) = \mathrm{ind}^G_P ( \tau ) $.

We are going to use a method similar to Bernstein's geometric lemma (see \cite{B-Z}). Recall that $H$ acts on the space $\mathcal{P}$ via right multiplications. For $\gamma\in G$, let $ \mathcal{O}_\gamma := P \gamma H \subseteq \CP $ be the $H$-orbit of $ P \gamma \in \CP$.  We label the $H$-orbits in $\mathcal{P}$ as $ \mathcal{O}_{ \gamma_0 }, \mathcal{O}_{ \gamma_1 }, \dots, \mathcal{O}_{ \gamma_m } $ such that
\begin{itemize}
    \item $ \mathcal{O}_{\gamma_0} $ is the unique open orbit;
    \item For each $ i = 0, 1, \dots, m - 1 $, $\mathcal{Y}_i := \bigcup_{j=0}^i \mathcal{O}_{ \gamma_j } $ is open in $ \mathcal{Y}_{i+1} $, where $ \mathcal{Y}_m = \mathcal{P} $.
\end{itemize}
For each $ i = 0, 1, \dots, m - 1 $, we have the following short exact sequence of representations of $H$:
\begin{equation}
    \label{gammaexact}
    0 \rightarrow \Gamma_c \left( \mathcal{Y}_{i}, \mathsf{E}|_{ \mathcal{Y}_{i} } \right) \to \Gamma_c \left( \mathcal{Y}_{i + 1}, \mathsf{E}|_{ \mathcal{Y}_{i + 1} } \right) \to \Gamma_c \left( \mathcal{O}_{ \gamma_{i + 1} }, \mathsf{E}|_{ \mathcal{O}_{ \gamma_{i + 1} } } \right) \rightarrow 0,
\end{equation}
where the first map is the extension by zero, and the second map is given by restriction. Moreover, there is an isomorphism
\begin{equation}
    \label{littleorbitsections}
    \Gamma_c \left( \mathcal{O}_{ \gamma_{i} }, \mathsf{E}|_{ \mathcal{O}_{ \gamma_{i} } } \right) \cong \mathrm{ind}^H_{ H_{ \gamma_i } } \left( \tau^{ \gamma_i } \otimes \rho_P^{ \gamma_i }|_{ H_{ \gamma_i } } \otimes \delta_{ H_{ \gamma_i } }^{-1/2} \right).
\end{equation}
For each $ i = 0, 1, \dots, m $, tensoring $ ( \ref{gammaexact} ) $ with the character $ \varphi_H^{-1} $, which preserves the exactness, and then taking homology yields a long exact sequence 
\begin{equation}
    \label{longexact}
    \begin{array}{llll}
        \cdots &\rightarrow \RH_1
        \left(
            H, \Gamma_c \left( \mathcal{O}_{ \gamma_{i} }, \mathsf{E}|_{ \mathcal{O}_{ \gamma_{i} } } \right) \otimes \varphi_H^{-1}
        \right) 
        &\rightarrow \RH_0
        \left(
            H, \Gamma_c \left( \mathcal{Y}_{i}, \mathsf{E}|_{ \mathcal{Y}_{i} } \right) \otimes \varphi_H^{-1}
        \right)
        &\\
        &\rightarrow \RH_0 
        \left(
            H, \Gamma_c \left( \mathcal{Y}_{i + 1}, \mathsf{E}|_{ \mathcal{Y}_{i + 1} } \right) \otimes \varphi_H^{-1}
        \right)
        &\rightarrow \RH_0
        \left(
            H, \Gamma_c \left( \mathcal{O}_{ \gamma_{i} }, \mathsf{E}|_{ \mathcal{O}_{ \gamma_{i} } } \right) \otimes \varphi_H^{-1}
        \right)
        &\rightarrow 0.
    \end{array}
\end{equation}
By Frobenius reciprocity, for each $ i = 0, 1, \dots, m $ and each $ j \ge 0 $,
\begin{equation}
    \label{shapiro}
    \begin{aligned}
         & \RH_j
        \left(
            H, \Gamma_c \left( \mathcal{O}_{ \gamma_{i} }, \mathsf{E}|_{ \mathcal{O}_{ \gamma_{i} } } \right) \otimes \varphi_H^{-1} 
        \right)                                                                          \\ 
        =& \RH_j
        \left(
            H, \mathrm{ind}^H_{ H_{ \gamma_i } } \left( \tau^{ \gamma_i } \otimes \rho_P^{ \gamma_i }|_{ H_{ \gamma_i } } \otimes \delta_{ H_{ \gamma_i } }^{-1/2} \right) \otimes \varphi_H^{-1} \right) & ( \text{by } ( \ref{littleorbitsections} ) )\\
        =& \RH_j
        \left(
            H_{ \gamma_i }, \tau^{ \gamma_i } \otimes \rho_P^{ \gamma_i }|_{ H_{ \gamma_i } } \otimes \delta_{ H_{ \gamma_i } }^{-1} \otimes \varphi_H^{-1}|_{ H_{ \gamma_i } }
        \right).
    \end{aligned}
\end{equation}
Assuming that $ ( \ref{targetspaces} ) $ holds, the long exact sequence $ ( \ref{longexact} ) $ together with isomorphisms $ ( \ref{shapiro} ) $ induces isomorphisms
\[
    \begin{aligned}
        \RH_i 
        \left(
            H, \Gamma_c \left( \mathcal{O}_{ \gamma_0 }, \mathsf{E}|_{ \mathcal{O}_{ \gamma_0 } } \right) \otimes \varphi_H^{-1}
        \right)
                 =& \RH_i
        \left(
            H, \Gamma_c \left( \mathcal{Y}_{1}, \mathsf{E}|_{ \mathcal{Y}_{1} } \right) \otimes \varphi_H^{-1}
        \right)\\
        = \cdots =& \RH_i
        \left(
            H, \Gamma_c \left( \mathcal{Y}_{i}, \mathsf{E}|_{ \mathcal{Y}_{i} } \right) \otimes \varphi_H^{-1}
        \right)\\
                 =& \RH_i
        \left(
            H, \Gamma_c \left( \mathcal{Y}_{i + 1}, \mathsf{E}|_{\mathcal{Y}_{i + 1} } \right) \otimes \varphi_H^{-1}
        \right)\\
        = \cdots =& \RH_i
        \left(
            H, \Gamma_c \left( \mathcal{P},\mathsf{E} \right) \otimes \varphi_H^{-1}
        \right)
    \end{aligned}
\]
for all $ i \ge 0 $. Recall that $ H_{ \gamma_0 } = H_0 $, under which $ \rho_P|_{H_0} = \delta_{H_0} = 1 $, $ \varphi_H|_{H_0} = \eta_H $, and $ \tau^{ \gamma_0 } = \tau $. We have 
\[
    \begin{aligned}
        \RH_i \left( H_0, \tau \otimes \eta_H^{-1} \right) =& \RH_i
        \left(
            H, \Gamma_c \left( \mathcal{O}_{ \gamma_0 }, \mathsf{E}|_{ \mathcal{O}_{ \gamma_0 } } \right) \otimes \varphi_H^{-1} \right) & ( \text{by } ( \ref{shapiro} ) ) \\
                                                           =& \RH_i
        \left(
            H, \Gamma_c \left( \mathcal{P}, \mathsf{E} \right) \otimes \varphi_H^{-1}
        \right)                                                                        \\
                                                           =& \RH_i
        \left(
            H, \Ind^G_P ( \tau ) \otimes \varphi_H^{-1}
        \right)
        .                                            & ( \text{by } ( \ref{indandsec} ) )
    \end{aligned}
\]
For our main theorem, only the case $i = 0$ is required. According to \cite[Theorem 1.1]{Prasad}, the dimension of the space 
\[
    \Hom_{H_0} \left( \tau, \eta_H \right) = \Hom_{H_0} ( \tau_1 \otimes \tau_2 \otimes \tau_3, \eta \circ \det )
\]
is at most $1$, if each $\tau_i$ is irreducible. Therefore, by the natural duality, we obtain
\[
    \Hom_{H_0} \left( \tau, \eta_H \right) \cong \RH_0 \left( H_0, \tau \otimes \eta_H^{-1} \right)^* = \RH_0 \left( H, \Ind^G_P ( \tau ) \otimes \varphi_H^{-1} \right)^* \cong \Hom_H \left( \pi_\tau, \varphi_H \right),
\]
where $ ( \cdot )^* $ denotes the dual space of a finite dimensional vector space. This completes the proof of Theorem $ \ref{mainthm1} $.

\subsection{Computation of the homology groups}

To prove $ ( \ref{targetspaces} ) $, we use the Hochschild-Serre spectral sequence. 
\begin{prpd}[{\cite[Theorem 5]{H}}]
    Let $Q$ be a locally compact totally disconnected group and $Q'$ be a normal subgroup of $Q$. Let $Q''$ denote the quotient $Q/Q'$. Then for any smooth representation $V$ of $Q$, we have the spectral sequence
    \[
        E_{p,q}^2 := \RH_p( Q'', \RH_q( Q', V ) ) \Rightarrow \RH_{p + q}( Q, V ).
    \]
\end{prpd}

By passing to the conjugate group $ P_\gamma = \gamma H_\gamma \gamma^{-1} $, it suffices to show that
\[
    \RH_i ( P_\gamma, W_\gamma ) := \RH_i 
    \left(
        P_\gamma, \tau \otimes \rho_P|_{ P_\gamma } \otimes { \left( \delta_{ H_\gamma }^{-1} \right) }^{ \gamma^{-1} } \otimes \left( \varphi_H^{-1} \right)^{ \gamma^{-1} }
    \right)
    =0
\]
for $ i = 0, 1 $. Recall that by the Levi decomposition, $ P = MN $, which induces a projection
\[
    \mathrm{pr}_M: P \rightarrow M, \quad
    \begin{bmatrix}
        g_1 & X   & Y  \\
            & g_2 & Z  \\
            &     & g_3
    \end{bmatrix}
    \mapsto
    \begin{bmatrix}
        g_1 &     &    \\
            & g_2 &    \\
            &     & g_3
    \end{bmatrix}
    ,
\]
with $ \ker \mathrm{pr}_M = N $. Set $ N_\gamma' := \ker( \mathrm{pr}_M|_{ P_\gamma } ) = N \cap \gamma H \gamma^{-1} $ and $ M_\gamma' := \mathrm{pr}_M ( P_\gamma ) \cong P_\gamma / N_\gamma' $. In what follows, we apply this spectral sequence to the groups $P_{\gamma}$, $N_\gamma'$ and $M_\gamma'$. 

\begin{lemd}
    For each orbit $ \mathcal{O}_\gamma $ and $ i \ge 0 $, the homology group $ \RH_i \left( N_\gamma', W_\gamma \right) $ vanishes for all $ i > 0 $, and also for $ i = 0 $ if the character $\psi_H$ is non-trivial on $N_\gamma$. More precisely, for $ \gamma \in \Gamma $,
    \[
        \RH_i \left( N_\gamma', W_\gamma \right) =
        \begin{cases}
            W_\gamma, & \text{if } i = 0 \text{ and } \gamma \in \Gamma',\\
            0,        & \text{otherwise}.
    \end{cases}
    \]
    Here $\Gamma$ and $\Gamma'$ are sets of representatives of orbits defined near Lemma $ \ref{truestabilizer} $.
\end{lemd}

\begin{proof}
    By \cite[Proposition 1.9]{B-Z}, the coinvariant functor $ ( \cdot )_{ N_\gamma } $ is exact, which implies the vanishing of homology groups $ H_i ( N_\gamma, \cdot ) $ for $ i  > 0 $. We turn to the case $ i = 0 $. Since $\tau$, $\rho_P$, and $ \delta_{ H_\gamma }^{ \gamma^{-1} } $ are all trivial on $N_\gamma'$ , $N_\gamma'$ acts on $W_\gamma$ via the character $ \psi_H^{ \gamma^{-1} } $. Suppose $\psi_H$ is not trivial on $N_\gamma$, and choose $ n_0 \in N_\gamma $ such that $ \psi_H ( n_0 ) \ne 1 $. Then we have
    \[
        \begin{aligned}
            W_\gamma \supseteq & \Span_\mathbb{C} \left\{ n' \cdot w - w \mid n' \in  N_\gamma', w \in W_\gamma \right\}               \\
                     \supseteq & \Span_\mathbb{C} \left \{ \left( \psi_H ( n_0 ) - 1 \right) w \mid w \in W_\gamma \right\} = W_\gamma,
        \end{aligned}
    \]
    which implies that 
    \[
        \RH_0 \left( N_\gamma', W_\gamma \right) = W_\gamma / \Span_\mathbb{C} \left\{ n' \cdot w - w \mid n' \in N_\gamma', w \in W_\gamma \right\} = 0.
    \]
    Conversely, if $ \psi_H|_{ N_\gamma } = 1 $, then $W_\gamma$ is a trivial representation of $N_\gamma'$, from which it follows that $ \RH_0 \left( N_\gamma', W_\gamma \right) = W_\gamma $. The second assertion is a direct consequence of Lemma $ \ref{truestabilizer} $.
\end{proof}

Combining this lemma with the Hochschild-Serre spectral sequence, the terms $ E_{p,q}^2 = H_p ( M_\gamma', H_q ( N_\gamma', W_\gamma ) ) $ vanish for all $ q > 0 $. Consequently, we obtain a natural isomorphism for each $ i \ge 0 $:
\[
    \RH_i \left( P_\gamma, W_\gamma \right) \cong \RH_i \left( M_\gamma', \RH_0 \left( N_\gamma', W_\gamma \right) \right),
\]
which vanishes except when $ i = 0 $ and the orbit $ \mathcal{O}_\gamma $ contains any of the following elements:
\[
    (15)(26) (=\gamma_0), (152643), (134625), (13)(25)(46).
\]
In order to calculate the homology groups $ \RH_i \left( P_\gamma, W_\gamma \right) $, which reduces to $ \RH_i \left( M_\gamma', W_\gamma \right) $ when $ \gamma \in \Gamma' \setminus \{ \gamma_0 \} $, we give the following parameterizations for $M_\gamma'$.

\begin{prpd}
    The groups $M_\gamma'$ for $ \gamma \in \Gamma' \setminus \{ \gamma_0 \} $ are described as follows:
    \begin{itemize}
        \item[(a)] For $ \gamma = (152643) $: $M_\gamma'$ consists of the matrices of the form
        \[
            \begin{bmatrix}
                g & gu &   &    &   &   \\
                  & g  &   &    &   &   \\
                  &    & h & hv &   &   \\
                  &    &   & h  &   &   \\
                  &    &   &    & g & hx\\
                  &    &   &    &   & h
            \end{bmatrix}
            =
            \begin{bmatrix}
                1 & u &   &   &   &  \\
                  & 1 &   &   &   &  \\
                  &   & 1 & v &   &  \\
                  &   &   & 1 &   &  \\
                  &   &   &   & 1 & x\\
                  &   &   &   &   & 1
            \end{bmatrix}
            \begin{bmatrix}
                g &   &   &   &   &  \\
                  & g &   &   &   &  \\
                  &   & h &   &   &  \\ 
                  &   &   & h &   &  \\
                  &   &   &   & g &  \\
                  &   &   &   &   & h
            \end{bmatrix}
        \]
        for $ g, h \in k^\times $ and $ u, v, x \in k $.
        \item[(b)] For $ \gamma = (134625) $: $M_\gamma'$ consists of the matrices of the form
        \[
            \begin{bmatrix}
                g & hx &   &    &   &   \\
                  & h  &   &    &   &   \\
                  &    & g & gu &   &   \\
                  &    &   & g  &   &   \\
                  &    &   &    & h & hv\\
                  &    &   &    &   & h
            \end{bmatrix}
            =
            \begin{bmatrix}
                1 & x &   &   &   &  \\
                  & 1 &   &   &   &  \\
                  &   & 1 & u &   &  \\
                  &   &   & 1 &   &  \\
                  &   &   &   & 1 & v\\
                  &   &   &   &   & 1
            \end{bmatrix}
            \begin{bmatrix}
                g &   &   &   &   &  \\
                  & h &   &   &   &  \\
                  &   & g &   &   &  \\
                  &   &   & g &   &  \\
                  &   &   &   & h &  \\
                  &   &   &   &   & h
            \end{bmatrix}
        \]
        for $ g, h \in k^\times $ and $ u, v, x \in k $.
        \item[(c)] For $ \gamma = (13)(25)(46) $: $M_\gamma'$ consists of the matrices of the form
        \[
            \begin{bmatrix}
                g & gu &   &    &   &   \\
                  & g  &   &    &   &   \\
                  &    & g & hx &   &   \\
                  &    &   & h  &   &   \\
                  &    &   &    & h & hv\\
                  &    &   &    &   & h
            \end{bmatrix}
            =
            \begin{bmatrix}
                1 & u &   &   &   &  \\
                  & 1 &   &   &   &  \\
                  &   & 1 & x &   &  \\
                  &   &   & 1 &   &  \\
                  &   &   &   & 1 & v\\
                  &   &   &   &   & 1
            \end{bmatrix}
            \begin{bmatrix}
                g &   &   &   &   &  \\
                  & g &   &   &   &  \\
                  &   & g &   &   &  \\
                  &   &   & h &   &  \\
                  &   &   &   & h &  \\
                  &   &   &   &   & h
            \end{bmatrix}
        \]
        for $ g, h \in k^\times $ and $ u, v, x \in k $.
    \end{itemize}
    Let $U_\gamma'$ and $T_\gamma'$ denote the groups consisting of the unipotent and diagonal matrices, respectively, occurring in the decomposition of $M_\gamma'$ above. It follows that
    \[
        M_\gamma' = T_\gamma' \ltimes U_\gamma' \cong ( k^\times )^2 \ltimes ( N_2 )^3
    \]
    for these permutations $\gamma$.
\end{prpd}

By a similar argument, we apply the Hochschild-Serre spectral sequence to the decomposition $ M_\gamma' = T_\gamma' \ltimes U_\gamma' $ and note that $ \RH_i \left( U_\gamma', W_\gamma \right) = 0 $ for $ i > 0 $. It follows that 
\[
    \RH_i \left( P_\gamma, W_\gamma \right) = \RH_i \left( M_\gamma', W_\gamma \right) = \RH_i \left( T_\gamma', \RH_0 ( U_\gamma', W_\gamma ) \right)
\]
for $ \gamma \in \Gamma' \setminus \{ \gamma_0 \} $ and $ i \ge 0 $. In these cases, $ \rho_P^{ \gamma }|_{ H_\gamma } \otimes \delta^{-1}_{ H_\gamma } = 1 $ by Proposition \ref{groupsforgamma}, and thus the representations $W_\gamma$ reduce to $ \tau \otimes \left( \varphi_H^{-1} \right)^{ \gamma^{-1} } $. For each representation $V$ of $\GL_2(k)$ and additive character $\xi$ of $k$, let $J_{\xi}(V)$ denote the twisted Jacquet module $ \RH_0 ( N_2, V \otimes \xi^{-1} ) $, then we obtain
\[
    \begin{aligned}
        \RH_0 \left( U_\gamma', W_\gamma \right)
        =& \RH_0 \left( U_\gamma', ( \tau_1 \otimes \tau_2 \otimes \tau_3 ) \otimes \left. \left( \varphi_H^{-1} \right)^{ \gamma^{-1} } \right|_{U_\gamma'} \right) \otimes \left. \left( \varphi_H^{-1} \right)^{ \gamma^{-1} } \right|_{T_\gamma'}\\
        =& \RH_0 \left( U_\gamma', \left( \tau_1 \otimes \psi^{-1} \right) \otimes \left( \tau_2 \otimes \psi^{-1} \right) \otimes \tau_3 \right) \otimes \left. \left( \varphi_H^{-1} \right)^{ \gamma^{-1} } \right|_{T_\gamma'}                                       \\
        =& \left( J_\psi ( \tau_1 ) \otimes J_\psi ( \tau_2 ) \otimes J ( \tau_3 ) \right) \otimes \left. \left( \varphi_H^{-1} \right)^{ \gamma^{-1} } \right|_{T_\gamma'},
    \end{aligned}
\]
where $N_2$ acts on $\psi$ via the isomorphism $ N_2 \cong k $.
For each $ \gamma \in \Gamma' \setminus \{ \gamma_0 \} $, we now establish a sufficient condition under which the homology groups $\RH_i \left( P_\gamma, W_\gamma \right) $ vanish.

\begin{lemd}
    Suppose that $ \gamma = (152643) $. If the space
    \[
        \Hom_{T_2}
        \left(
            J( \tau_3 )_\mathrm{ss}, ( \eta \circ \det ) \otimes \left( \omega_{ \tau_1 }^{-1} \otimes \omega_{ \tau_2 }^{-1} \right)
        \right)=0,
    \]
    then 
    \[
        \RH_i \left( P_\gamma, W_\gamma \right) = 0, \quad i \ge 0.
    \]
\end{lemd}

\begin{proof}
    Note that $T_\gamma'$ acts on $ J( \tau_1 ) \otimes J( \tau_2 ) \otimes J( \tau_3 ) $ by
    \[
        \tau_1 \left(
        \begin{bmatrix}
            g &  \\
              & g
        \end{bmatrix}
        \right) \otimes \tau_2 \left( 
        \begin{bmatrix}
            h &  \\
              & h
        \end{bmatrix}
        \right) \otimes \tau_3 \left(
        \begin{bmatrix}
            g &  \\
              & h
        \end{bmatrix}
        \right)
    \]
    and acts on $ \varphi_H^{ \gamma^{-1} } $ by $ \eta(gh) $. We have isomorphisms for all $ i \ge 0 $:
    \[
        \begin{aligned}
            \RH_i
            \left(
                T_\gamma', \RH_0 ( U_\gamma', W_\gamma )
            \right)
            \cong J_\psi ( \tau_1 ) 
            &\otimes J_\psi ( \tau_2 )\\
            &\otimes \RH_i
            \left(  
                T_2, J( \tau_3 ) \otimes \left( \omega_{ \tau_1 } \otimes \omega_{ \tau_2 } \right) \otimes \left( \eta^{-1} \circ \det \right) 
            \right).
        \end{aligned}
    \]
    Now consider the homology group of $T_2$. Arguing using the long exact sequence in homology, it is sufficient to show that the homology group remains zero if we replace $ J( \tau_3 ) $ by its semi-simplification $ J( \tau_3 )_\mathrm{ss} $. Write
    \[
        J( \tau_3 )_\mathrm{ss} = 
        \begin{cases}
            ( \xi_{11} \otimes \xi_{12} ) \oplus ( \xi_{21} \otimes \xi_{22} ), & \text{if } \tau_3 \text{ is a principal series representation,}\\
            \xi_1 \otimes \xi_2,                                                & \text{if } \tau_3 \text{ is a Steinberg representation,}       \\
            0,                                                                  & \text{if } \tau_3 \text{ is a supercuspidal representation.}
        \end{cases}
    \]
    In the case where $\tau_3$ is a principal series representation, we have
    \begin{equation}
        \label{ssspace}
        \begin{aligned}
             & \RH_i \left( T_2, J( \tau_3 )_\mathrm{ss} \otimes \left( \omega_{ \tau_1 } \otimes \omega_{ \tau_2 } \right) \otimes \left( \eta^{-1} \circ \det \right) \right)\\
            =& \bigoplus_{j = 1}^2 \RH_i \left( T_2, \left( \xi_{j1} \omega_{ \tau_1 } \eta^{-1} \right) \otimes \left( \xi_{j2} \omega_{ \tau_2 } \eta^{-1} \right) \right)   \\
            =& \bigoplus_{ \alpha + \beta = i } \bigoplus_{j = 1}^2 \RH_\alpha \left( k^\times, \xi_{j1} \omega_{ \tau_1 } \eta^{-1} \right) \otimes \RH_\beta \left( k^\times, \xi_{j2} \omega_{ \tau_2 } \eta^{-1} \right),
        \end{aligned}
    \end{equation}
    where the last equation follows from the K\"{u}nneth formula. Recall that for each character $\chi$ of $k^\times$,
    \[
        \RH_i ( k^\times, \chi ) =
        \begin{cases}
            \mathbb{C}, & \text{if } i = 0, 1 \text{ and } \chi = 1,\\
            0,          & \text{otherwise.}
        \end{cases}
    \]
    A straightforward calculation shows that
    \[
        ( \ref{ssspace} ) = 
        \begin{cases}
            \mathbb{C} \oplus \mathbb{C},                                     & \text{if } \xi_{l1} \omega_{ \tau_1 } = \xi_{l2} \omega_{ \tau_2 } = \eta \text{ and } i = 0, 2,                                                                                \\
            \mathbb{C} \oplus \mathbb{C} \oplus \mathbb{C} \oplus \mathbb{C}, & \text{if } \xi_{l1} \omega_{ \tau_1 } = \xi_{l2} \omega_{ \tau_2 } = \eta \text{ and } i = 1,\\
            0,                                                                & \text{otherwise.}
        \end{cases}
    \]
    Consequently, the vanishing of $ ( \ref{ssspace} ) $ for $ i = 0 $ guarantees its vanishing for all $ i \ge 0 $.
    Therefore, if $\tau_3$ is a principal series representation and 
    \[
        \Hom_{T_2}
        \left(
            J( \tau_3 )_\mathrm{ss},( \eta \circ \det ) \otimes \left( \omega_{ \tau_1 }^{-1} \otimes \omega_{ \tau_2 }^{-1} \right)
        \right)
        = 0,
    \]
    then by taking duality, we have
    \[
        \RH_0
        \left(
            T_2, J( \tau_3 )_\mathrm{ss} \otimes \left( \omega_{ \tau_1 } \otimes \omega_{ \tau_2 } \right) \otimes \left( \eta^{-1} \circ \det \right)
        \right)
        = 0.
    \]
    It follows that the spaces
    \[
        \RH_i
        \left(
            T_2, J( \tau_3 ) \otimes \left( \omega_{ \tau_1 } \otimes \omega_{ \tau_2 } \right) \otimes \left( \eta^{-1} \circ \det \right)
        \right)
    \]
    also vanish for all $ i \ge 0 $. By a similar argument, this assertion also holds when $\tau_3$ is a Steinberg representation, and holds automatically when $\tau_3$ is supercuspidal.
\end{proof}

The following analogous results hold for $ \gamma = (134625) $ and $ \gamma = (13)(25)(46) $, with proofs omitted.

\begin{lemd}
    Suppose that $ \gamma = (134625) $. If the space:
    \[
        \Hom_{T_2}
        \left(
            J( \tau_1 )_\mathrm{ss}, ( \eta \circ \det ) \otimes \left( \omega_{ \tau_2 }^{-1} \otimes \omega_{ \tau_3 }^{-1} \right)
        \right)
        = 0,
    \]
    then 
    \[
        \RH_i \left( P_\gamma, W_\gamma \right) = 0, \quad i \ge 0.
    \]
\end{lemd}

\begin{lemd}
    Suppose that $ \gamma = (13)(25)(46) $. If the space
    \[
        \Hom_{T_2}
        \left(
            J( \tau_2 )_\mathrm{ss}, ( \eta \circ \det ) \otimes \left( \omega_{ \tau_1 }^{-1} \otimes \omega_{ \tau_3 }^{-1} \right)
        \right)
        = 0, 
    \]
    then 
    \[
        \RH_i \left( P_\gamma, W_\gamma \right) = 0, \quad i \ge 0.
    \]
\end{lemd}
Recall the definition of good representations (before Theorem \ref{mainthm1}). If $ \tau = \tau_1 \otimes \tau_2 \otimes \tau_3 $ is good, then it satisfies all the conditions in the three preceding lemmas. Therefore, for any orbit $ \mathcal{O}_\gamma $ that is not the unique open orbit, $ ( \ref{targetspaces} ) $ holds, and thus Theorem $ \ref{mainthm1} $ is proved according to the analysis at the end of the last subsection.

\subsection{Criterion for good representations}

In this subsection we prove Proposition \ref{goodcr}, which provides a criterion for when a representation $ \tau = \tau_1 \otimes \tau_2 \otimes \tau_3 $ of $M$ is good, where each $\tau_i$ is an irreducible admissible smooth representation of $\GL_2(k)$. We prove this criterion by direct computations of the spaces
\[
    \Hom_{T_2}
    \left(
        J( \tau_i )_\mathrm{ss} ( \eta \circ \det ) \otimes \left( \omega_{ \tau_j }^{-1} \otimes \omega_{ \tau_l }^{-1} \right)
    \right)
\]
for $ i = 1, 2, 3 $ and $ j, l \in \{ 1, 2, 3 \} \setminus \{ i \} $ with $ j < l $. According to \cite[Restriction-Induction Lemma (in section 9)]{B-H}, the semi-simplifications $ J( \tau_i )_\mathrm{ss} $ are listed below by cases.                     
\begin{itemize}
    \item If $\tau_i$ is supercuspidal, then 
    \[
        J( \tau_{i} )_\mathrm{ss} = J( \tau_i ) = 0.
    \]
    \item If $ \tau_i = \chi_i \circ \det $ for some character $\chi_i$, then 
    \[
        J( \tau_{i} )_\mathrm{ss} = J( \tau_i ) = \chi_i \otimes \chi_i.
    \]
    \item If $ \tau_i = \mathrm{St} ( \chi_i ) $ for some character $\chi_i$, then 
    \[
        J( \tau_{i} )_\mathrm{ss} = J( \tau_i ) =( \chi_i \otimes \chi_i ) \otimes \delta_{B_2} = \chi_i \absk{ \cdot } \otimes \chi_i \absk{ \cdot }^{-1}.
    \]
    \item If $ \tau_i = \Ind^{ \GL_2(k) }_{B_2} ( \chi_{ij} \otimes \chi_{il} ) $ for characters $\chi_{ij}$ and $\chi_{il}$, then 
    \[
        \begin{aligned}
            J( \tau_{i} )_\mathrm{ss} =& ( \chi_{ij} \otimes \chi_{il} ) \otimes \delta_{B_2}^{1/2} \oplus ( \chi_{il} \otimes \chi_{ij} ) \otimes \delta_{B_2}^{1/2}\\
                                      =& \left( \chi_{ij} \absk{ \cdot }^{1/2} \otimes \chi_{il} \absk{ \cdot }^{-1/2} \right) \oplus \left( \chi_{il} \absk{ \cdot }^{1/2} \otimes \chi_{ij} \absk{ \cdot }^{-1/2} \right).
        \end{aligned}
    \]
\end{itemize}
Assume first that $ \tau_i = \chi_i \circ \det $ or $ \tau_i = \mathrm{St} ( \chi_i ) $. In both cases $ J( \tau_i )_\mathrm{ss} $ is one-dimensional. Therefore, the vanishing of the Hom-space reduces to a comparison of characters on $T_2$. Specifically,  
\[
    \Hom_{T_2}
    \left(
        J( \tau_i )_\mathrm{ss},( \eta \circ \det ) \otimes \left( \omega_{ \tau_j }^{-1} \otimes \omega_{ \tau_l }^{-1} \right)
    \right)
    = 0
\]
if and only if:
\begin{itemize}
    \item $ \chi_i \otimes \chi_i \ne ( \eta \circ \det ) \otimes \left( \omega_{ \tau_j }^{-1} \otimes \omega_{ \tau_l }^{-1} \right) $ for $ \tau_i = \chi_i \circ \det $;
    \item $ \chi_i \otimes \chi_i \ne ( \eta \circ \det ) \otimes \left( \omega_{ \tau_j }^{-1} \otimes \omega_{ \tau_l }^{-1} \right) \otimes \delta_{B_2}^{-1}$ for $ \tau_i = \mathrm{St} ( \chi_i ) $.
\end{itemize}
Suppose that $ \tau_i = \Ind^{ \GL_2(k) }_{B_2} ( \chi_{ij} \otimes \chi_{il} ) $ is irreducible. By the decomposition of $ J( \tau_i )_\mathrm{ss} $, the target Hom-space vanishes if and only if the following two Hom-spaces vanish simultaneously:
\[
    \Hom_{T_2}
    \left(
        ( \chi_{ij} \otimes \chi_{il} ) \otimes \delta_{B_2}^{1/2},( \eta \circ \det ) \otimes \left( \omega_{ \tau_j }^{-1} \otimes \omega_{ \tau_l }^{-1} \right)
    \right)
\]
and
\[
    \Hom_{T_2}
    \left(
        ( \chi_{il} \otimes \chi_{ij} ) \otimes \delta_{B_2}^{1/2}, ( \eta \circ \det ) \otimes \left( \omega_{ \tau_j }^{-1} \otimes \omega_{ \tau_l }^{-1} \right)
    \right),
\]
which is equivalent to 
\[
    \chi_{ij} \otimes \chi_{il} \ne ( \eta \circ \det ) \otimes \left( \omega_{ \tau_j }^{-1} \otimes \omega_{ \tau_l }^{-1} \right) \otimes \delta_{B_2}^{-1/2}
\]
and
\[
    \chi_{ij} \otimes \chi_{il} \ne ( \eta \circ \det ) \otimes \left( \omega_{ \tau_l }^{-1} \otimes \omega_{ \tau_j }^{-1} \right) \otimes \delta_{B_2}^{1/2}.
\]

\section{Proofs of Corollary \ref{mainthm2} and Corollary \ref{mainthm4}}

Let $ \tau = \tau_1 \otimes \tau_2 \otimes \tau_3 $ be a good representation. Now we consider the dimension of the space
\[
    \Hom_H ( \pi_\tau, \varphi_H ) = \Hom_{H_0} \left( \tau_1 \otimes \tau_2 \otimes \tau_3, \eta \circ \det \right).
\]
In this section, we further assume that $\tau_1$, $\tau_2$, and $\tau_3$ are infinite-dimensional, and the product of central characters satisfies $ \omega_{ \tau_1 } \omega_{ \tau_2 } \omega_{ \tau_3 } = \eta^2 $. For the case where $\tau_1$, $\tau_2$, and $\tau_3$ are not all essential discrete series representations, the nonvanishing criterion relies on \cite[Theorem 1.2]{Prasad}.

\begin{prpd}[{\cite[Theorem 1.2]{Prasad}}]
    Let $\tau_1$, $\tau_2$, and $\tau_3$, be infinite-dimensional, irreducible, admissible smooth representations of $\GL_2(k)$ with central characters satisfying $ \omega_{\tau_1} \omega_{\tau_2} \omega_{\tau_3} = \eta^2 $. Suppose that not all of them are essential discrete series representations. Then
    \[
        \dim \Hom_{ H_0 } \left( \tau_1 \otimes \tau_2 \otimes \tau_3, \eta \circ \det \right) = 1.
    \]
\end{prpd}

This immediately yields Corollary \ref{mainthm2}.

Before proceeding to the proof of Corollary \ref{mainthm4}, we recall some basic facts concerning the Langlands correspondence and $\varepsilon$-factors. Let $W_k$ denote the Weil group of $k$ and let $ W_k' = W_k \times \SL_2 ( \mathbb{C} ) $ be the Weil-Deligne group of $k$. For a representation $\sigma$ of $W_k'$ and an additive character $\psi$ of $k$, we define the $\varepsilon$-factor $ \varepsilon( \sigma, s, \psi ) $ as in \cite{Tate}. In the rest of this paper, we adopt the abbreviation $\varepsilon( \sigma, \psi ) := \varepsilon( \sigma, 1/2, \psi ) $ for convenience. If the representation $\sigma$ satisfies that the value of $ \varepsilon( \sigma, \psi ) $ is independent of the choice of the additive character $\psi$, then we simply write $ \varepsilon( \sigma ) $. For example, if $ \det \sigma = 1 $, then by \cite[8.1.5]{Prasad},
\[
    \varepsilon( \sigma, \psi_a ) = ( \det \sigma )^{ \dim ( \sigma ) }( a ) \varepsilon( \sigma, \psi ), \quad a \in k,
\]
where $ \psi_a ( x ) = \psi( ax ) $ for all $ x \in k $. It follows that $ \varepsilon( \sigma, \psi ) $ is independent of the choice of $\psi$.

The proofs of Corollary \ref{mainthm4} and Theorem \ref{mainthm3} rely on the following theorem in \cite{Prasad} 

\begin{thmd}[{\cite[Theorem 1.4]{Prasad}}]
    \label{thm1.4}
    Let $\tau_1$, $\tau_2$, and $\tau_3$ be infinite-dimensional, irreducible, admissible smooth representations of $\GL_2(k)$ such that the product of their central characters is trivial. If all the representations $\tau_i$, for $ i = 1, 2, 3 $, are supercuspidal, assume that the residue characteristic of $k$ is odd. Let $\sigma_i$ be the representation of $W_k'$ corresponding to $\tau_i$. Then    
    \[
        \varepsilon( \sigma_1 \otimes \sigma_2 \otimes \sigma_3 ) = \pm1.
    \]
    The multiplicity of the trilinear forms is determined by the $\varepsilon$-factors:
    \[
        \dim \Hom_{ H_0 } \left( \tau_1 \otimes \tau_2 \otimes \tau_3, \mathbb{C} \right) =
        \begin{cases}
            1, & \text{if } \varepsilon( \sigma_1 \otimes \sigma_2 \otimes \sigma_3 ) = 1, \\
            0, & \text{if } \varepsilon( \sigma_1 \otimes \sigma_2 \otimes \sigma_3 ) = -1.
        \end{cases}
    \]   
\end{thmd}

In the remainder of the paper, we further assume that $k$ is a local non-archimedean field with odd residue characteristic. Let $\mathrm{St}$ denote the representation $\mathrm{St}(1)$, where $1$ is the trivial character of $k^\times$. For any character $\chi$ of $k^\times$, we have $ \mathrm{St} ( \chi ) = ( \chi \circ \det ) \otimes \mathrm{St} $. Under the Langlands correspondence, $\mathrm{St}$ corresponds to the special representation $\mathrm{sp}(2)$ of the Weil-Deligne group $W_k'$ (see \cite[Theorem 33.3]{B-H}).

\begin{prpd}
    \label{specialcr}
    Let $\tau_1$, $\tau_2$, and $\tau_3$ be infinite-dimensional, essential discrete series representations of $\GL_2(k)$ with central characters satisfying $ \omega_{ \tau_1 } \omega_{ \tau_2 } \omega_{ \tau_3 } = \eta^2 $. Suppose that one of the following conditions is satisfied:
    \begin{itemize}
        \item There is exactly one Steinberg representation. Suppose it is $ \tau_1 = \mathrm{St} \left( \chi_1 \right) $ for a character $\chi_1$ of $k^\times$ while $\tau_2$ and $\tau_3$ are supercuspidal satisfying 
        \[
            \left( \chi_1 \eta^{-1} \circ \det \right) \otimes \tau_2 \ncong \tau_3^\vee.
        \]
        \item There are exactly two Steinberg representations among $\tau_1$, $\tau_2$, and $\tau_3$.
    \end{itemize}
    Then
    \[
        \dim \Hom_{ H_0 } ( \tau_1 \otimes \tau_2 \otimes \tau_3, \eta \circ \det ) = 1.
    \]
\end{prpd}

\begin{proof}
    Let $\sigma_i$ denote the representation of $W_k'$ corresponding to $\tau_i$ for each $ i = 1, 2, 3 $. Assume  first that $ \tau_1 = \mathrm{St} \left( \chi_1 \right) $ with $\tau_2$ and $\tau_3$ supercuspidal. Define the representation $ \tau_2' := \left( \chi_1 \eta^{-1} \circ \det \right) \otimes \tau_2 $ and let $\sigma_2'$ denote its corresponding representation of $W_k'$. Set
    \[
        \varepsilon_0 := \varepsilon \left( \sigma_1 \otimes \sigma_2' \otimes \sigma_3 \right) = \varepsilon \left( \mathrm{sp}(2) \otimes \sigma_2' \otimes \sigma_3 \right).
    \]
    By Theorem \ref{thm1.4}, we have $ \dim \Hom_{ H_0 } \left( \tau_1 \otimes \tau_2 \otimes \tau_3, \eta \circ \det \right) = 1 $ if and only if $ \varepsilon_0 = 1 $. We now apply \cite[Proposition 8.5]{Prasad} to $ \mathrm{St} \otimes \tau_2' \otimes \tau_3 $, which shows that $ \varepsilon_0 = 1 $ if and only if $ \sigma_2' \ncong \sigma_3^\vee $, or equivalently, $ \tau_2' = \left( \chi_1 \eta^{-1} \circ \det \right) \otimes \tau_2 \ncong \tau_3^\vee $.

    If at least two representations of $\tau_1$, $\tau_2$, and $\tau_3$ are Steinberg representations, then without loss of generality, we assume $ \tau_1 = \mathrm{St} \left( \chi_1 \right) $ and $ \tau_2 = \mathrm{St} \left( \chi_2 \right) $. Define $ \tau_3' := \left( \chi_1 \chi_2 \eta^{-1} \circ \det \right) \otimes \tau_3 $ and let $\sigma_3'$ be its corresponding representation of $W_k'$. We have
    \[
        \varepsilon_0 = \varepsilon \left( \mathrm{sp}(2) \otimes \mathrm{sp}(2) \otimes \sigma_3' \right).
    \]
    Then the application of \cite[Proposition 8.6]{Prasad} shows that $ \varepsilon_0 = 1 $ if and only if 
    \[
        \tau_3' = \left( \chi_1 \chi_2 \eta^{-1} \circ \det \right) \otimes \tau_3 \ncong \mathrm{St}.
    \]
    If $\tau_3$ is supercuspidal, then so is $\tau_3'$. The above condition is automatically satisfied. If $ \tau_3 = \mathrm{St} \left( \chi_3 \right) $ for some character $\chi_3$ of $k^\times$, then
    \[
        \tau_3' = \left( \chi_1 \chi_2 \chi_3 \eta^{-1} \circ \det \right) \otimes \mathrm{St} = \mathrm{St}.
    \]
    To summarize, $ \varepsilon_0 = 1 $ when there are exactly two Steinberg representations among $\tau_1$, $\tau_2$, and $\tau_3$, and $ \varepsilon_0 \ne 1 $ when all three are Steinberg representations.
\end{proof}
    
By combining Theorem \ref{mainthm1} and Proposition \ref{specialcr}, we conclude the proof of Corollary \ref{mainthm4}.

\section{Proof of Theorem \ref{mainthm3}}

In this section, we assume further that the residual characteristic of $k$ is odd. We consider the case where the representations $\tau_1$, $\tau_2$, and $\tau_3$ are all supercuspidal with the product $ \omega_{ \tau_1 } \omega_{ \tau_2 } \omega_{ \tau_3 } $ of central characters equal to $\eta^2$. By the good criterion (Proposition \ref{goodcr}), $ \tau = \tau_1 \otimes \tau_2 \otimes \tau_3 $ is automatically a good representation, and so 
\[
    \Hom_{H} ( \pi_\tau, \varphi_H ) = \Hom_{ H_0 } ( \tau_1 \otimes \tau_2 \otimes \tau_3, \eta \circ \det ).
\]
For each $ i = 1, 2, 3 $, define $ \tau_i' := \left( \eta^{-1} \circ \det \right) \otimes \tau_i $ and denote the representations of $W_k$ which correspond to $\tau_i$, $\tau_i'$ under the Langlands correspondence by $\sigma_i$, $\sigma_i'$, respectively. By Theorem \ref{thm1.4}, the space of trilinear forms is nontrivial if and only if the $\varepsilon$-factor
\[
    \varepsilon_0 := \varepsilon \left( \sigma_1' \otimes \sigma_2 \otimes \sigma_3 \right) = \varepsilon \left( \sigma_1 \otimes \sigma_2' \otimes \sigma_3 \right) = \varepsilon \left( \sigma_1 \otimes \sigma_2 \otimes \sigma_3' \right)
\]
equals $1$. Before proceeding to the computation of $ \varepsilon_0 $, we briefly review some preliminaries on representations of Weil groups of local non-archimedean fields and on local $\varepsilon$-factors.

\subsection{Preliminaries}

\begin{itemize}
    \item For the local field $k$ with odd residue characteristic, there are three quadratic (separable) extensions up to isomorphism. Let $ \varpi \in k$ be a uniformizing parameter and let $ u \in \mathcal{O}_k^\times $ be a unit whose image is not a square in the residue field. Then the three quadratic extensions are precisely the following fields:
    \[
        K_1 = k( \sqrt u ), K_2=k( \sqrt\varpi ), K_3 = k( \sqrt{ u\varpi } ).
    \]
    Among them, $K_1/k$ is unramified and the others are tamely ramified.
    \item Let $ L = K_1 K_2 = K_1 K_3 = K_2 K_3 = k( \sqrt{u}, \sqrt{ \varpi } ) $. Then the Galois group
    \[
        \Gal ( L/k ) =\{ 1, \phi_1, \phi_2, \phi_3 = \phi_1 \phi_2 \} \cong \mathbb{Z} / 2\mathbb{Z} \oplus \mathbb{Z} / 2\mathbb{Z},
    \]
    where $ \phi_1 \left( \sqrt{u} \right) = - \sqrt{u} $ and $ \phi_2 \left( \sqrt{ \varpi } \right) = -\sqrt{ \varpi } $. For $ i = 1, 2, 3 $, the Weil groups $ W_{ K_i } $ are index-$2$ subgroups of $W_k$, and $W_L$ is an index-$2$ subgroup of each $ W_{ K_i } $. For any distinct $ i, j \in \{ 1, 2, 3 \} $, we have
    \[
        W_{ K_i } W_{ K_j } = W_F \text{ and } W_{ K_i } \cap W_{ K_j } = W_L.
    \]
    \item According to the Langlands correspondence for $\GL_2$, there is a one-to-one correspondence between supercuspidal representations of $\GL_2(k)$ and two-dimensional irreducible induced representations of $W_k$. Such a representation is induced from a character of $K^\times$ (for some separable quadratic extension $K/k$) that does not factor through the norm map $ \mathrm{N}^K_k $. In the remainder of this paper, we will denote the representation $ \mathrm{Ind}^{ W_k }_{ W_K } \xi $ by $ \mathrm{Ind}^k_K \xi $ for short.
    \item For a separable finite extension of local fields $E/F$ and a nontrivial additive character $\psi$ of $F$, we define the Langlands constant of $\psi$ corresponding to $E/F$ by
    \[
        \lambda_{ E/F } ( \psi ) := \frac{ \varepsilon( \mathrm{Ind}^F_E \xi, s, \psi ) }{ \varepsilon( \xi, s, \psi_E ) },
    \]
    where we set $ \psi_E := \psi \circ \tr^E_F$. This definition of $\psi_E$ will remain in effect throughout this paper. The Langlands constant is independent of both the character $\xi$ and $ s \in \mathbb{C} $. Furthermore, for any $n$-dimensional representation $\sigma$ of $W_E$, we have
    \[
        \frac{ \varepsilon \left( \mathrm{Ind}^F_E \sigma, s, \psi \right) }{ \varepsilon( \sigma, s, \psi_E ) } = \lambda_{E/F} ( \psi )^n
    \]
    (see \cite[Corollary 30.4, 34.3]{B-H}).
    \item Continuing with the notation, let $ \varsigma_{ E/F } $ be the character $ \det \left( \mathrm{Ind}^F_E 1_E \right) $, where $1_E$ is the trivial character of $E^\times$ and $ \mathrm{Ind}^F_E 1_E: W_F \rightarrow \GL_{ [E:F] } ( \mathbb{C} ) $. The central character of a representation of $\GL_2(F)$ corresponding to a representation $\sigma$ of $W_F$ equals $ \det \sigma $ (see \cite[Proposition 33.4]{B-H}). Therefore, the supercuspidal representation corresponding to $ \mathrm{Ind}^F_E \xi $ has central character
    \[
        \det \left( \mathrm{Ind}^F_E \xi \right) = \left( \xi|_{ F^\times } \right) \det \left( \mathrm{Ind}^F_E 1_E \right) = \left( \xi|_{ F^\times } \right) \varsigma_{ E/F }
    \]
    (see \cite[34.2]{B-H}). Since the image of $ \mathrm{Ind}^F_E 1_E $ in $ \GL_{ [E:F] } ( \mathbb{C} ) $ consists of symmetric matrices, we have 
    \[
        \varsigma_{ E/F }(w) = \det ( \mathrm{Ind}^F_E 1_E (w) ) = \pm1, \quad \text{for all } \omega \in W_E,
    \]
    and thus $ \varsigma_{ E/F }^2 = 1 $. By \cite[30.4.3]{B-H}, $ \lambda_{ E/F }( \psi )^2 = \varsigma_{ E/F } (-1) $, which implies that $ \lambda_{ E/F } ( \psi ) $ is a fourth root of unity.
    \item Now assume further that $F$ is a local non-archimedean field and $E/F$ is Galois. Let $\sigma$ be a representation of $W_E$ and $ \phi \in \Gal( E/F ) $. Consider the twisted representation $ \sigma^\phi = \sigma \circ \Ad( \phi )$. By \cite[1.17]{Tunnell}, we have
    \begin{equation}
        \label{twistepsilon}
        \varepsilon \left( \sigma^\phi, \psi_E \right) = \varepsilon \left( \sigma, \psi_E \right).
    \end{equation}
    Let $ \sigma^\vee $ denote the contragredient of $\sigma$. Then by \cite[1.1.6]{Tunnell}, we have
    \begin{equation}
        \label{contraepsilon}
        \varepsilon( \sigma, \psi_E ) \varepsilon \left( \sigma^\vee, \psi_E \right) = \det ( \sigma ) (-1).
    \end{equation}
\end{itemize}
The following lemma is a direct consequence of the properties stated above.

\begin{lemd}
    \label{varsigma-11}
    Let $L$, $K_1$, $K_2$, $K_3$, and $k$ be as above. Then for each $ i = 1, 2, 3 $, we have
    \[
        \varsigma_{ L/k } (-1) = \varsigma_{ L/K_i } (-1) = 1.
    \]
\end{lemd}

\begin{proof}
    Consider the image of the homomorphism $ \mathrm{Ind}^k_L 1_L: W_k \rightarrow \GL_4( \mathbb{C} )$. It consists of the following matrices
    \[
        1_4,
        \begin{bmatrix}
            0   & 1_2\\
            1_2 & 0
        \end{bmatrix}
        ,
        \begin{bmatrix}
            w_2 & 0  \\
            0   & w_2
        \end{bmatrix}
        , w_4,
    \]
    all of which have trivial determinants. It follows that $ \varsigma_{ L/k } $ is trivial. Let $\xi$ be any character of $L^\times$ and $\psi$ be a nontrivial additive character of $k$. Then we have
    \[
        \varepsilon \left( \mathrm{Ind}^k_L \xi, \psi \right) = \varepsilon( \xi, \psi_L ) \lambda_{ L/k } ( \psi ).
    \]
    On the other hand, for each $ i = 1, 2, 3 $,
    \[
        \begin{aligned}
            \varepsilon \left( \mathrm{Ind}^k_L \xi, \psi \right) =& \varepsilon \left( \mathrm{Ind}^{k}_{ K_i } \mathrm{Ind}^{ K_i }_{L} \xi, \psi \right) = \varepsilon \left( \mathrm{Ind}^{K_i}_L \xi, \psi_{K_i} \right) \lambda_{ K_i/k }^2 ( \psi )   \\
                                                                  =& \varepsilon( \xi, \psi_L ) \lambda_{ L/K_i } \left( \psi_{ K_i } \right) \lambda_{ K_i/K }^2 ( \psi ).
        \end{aligned}
    \]
    Comparing both expressions yields 
    \[
        \lambda_{ L/k } ( \psi ) = \lambda_{ L/K_i } \left( \psi_{ K_i } \right) \lambda_{ K_i/K }^2 ( \psi ),
    \]
    and hence
    \[
        \varsigma_{ L/K_i } (-1) = \lambda^2_{ L/K_i } ( \psi_{ K_i } ) = \lambda^2_{ L/k } ( \psi ) = \varsigma_{ L/k } (-1) = 1.
    \]
\end{proof}
    
For $ i = 1, 2, 3 $, we denote by $ k(\tau_i) $ the quadratic extension of $k$ such that, under the Langlands correspondence, the representation corresponding to $\tau_i$ is induced from a character of $ k( \tau_i )^\times $. The proof of Theorem \ref{mainthm3} is divided into three cases: 
\begin{itemize}
    \item Case $ ( F_1, F_2, F_3 ) $: The fields $ k( \tau_1 ) $, $ k( \tau_2 ) $, and $ k( \tau_3 ) $ are pairwise distinct, and denoted by $F_1$, $F_2$, and $F_3$, respectively.
    \item Case $ ( F_1, F_1, F_2 ) $: Exactly two of the fields $ k( \tau_1 ) $, $ k( \tau_2 ) $, and $ k( \tau_3 ) $ coincide. Without loss of generality, we may assume that $ k( \tau_1 ) = k( \tau_2 ) $, denote the common field by $F_1$, and set $ k( \tau_3 ) := F_2 $.
    \item Case $ ( F, F, F ) $: All three fields coincide: $ k( \tau_1 ) = k( \tau_2 ) = k( \tau_3 ) $. We denote this common field by $F$.  
\end{itemize}
 
\subsection{Case $(F_1,F_2,F_3)$}

The following proposition is a corollary of \cite[Proposition 8.7]{Prasad}.
    
\begin{prpd}
    \label{cr1}
    Let $\tau_1$, $\tau_2$, and $\tau_3$ be irreducible supercuspidal representations of $\GL_2(k)$ with central characters satisfying $ \omega_{ \tau_1 } \omega_{ \tau_2 } \omega_{ \tau_3 } = \eta^2 $. Suppose that the representation $\tau_i$ corresponds to $ \mathrm{Ind}^k_{ F_i } \xi_i $ for each $ i = 1, 2, 3 $, where $F_1$, $F_2$, and $F_3$ are three distinct quadratic extension of $k$, and $\xi_i$ is a character of $F_i$. Then 
    \[
        \dim \Hom_{ H_0 } \left( \tau_1 \otimes \tau_2 \otimes \tau_3 , \eta \circ \det \right) = 1.
    \]
\end{prpd}

\begin{proof}
    Recall that $ \tau_1' = \tau_1 \otimes \left( \eta^{-1} \circ \det \right) $ is a supercuspidal representation corresponding to the representation $ \mathrm{Ind}^k_{ F_1 } \xi_1 \left( \eta^{-1} \circ \mathrm{N}^{ F_1 }_k \right) $ with central character $ \omega_{ \tau_1' } = \xi_1|_{ k^\times } \eta^{-2} \varsigma_{ F_1/k } = \omega_{ \tau_1 } \eta^{-2} $. Therefore, the product of central characters of $\tau_1', \tau_2, \tau_3$ is trivial. By \cite[Proposition 8.7]{Prasad}, the $\varepsilon$-factor
    \[
        \varepsilon_0 = \varepsilon( \sigma_{1}' \otimes \sigma_2 \otimes \sigma_3 ) = 1.
    \]
\end{proof}

\subsection{Case $(F_1,F_1,F_2)$}

The following facts from \cite{Tunnell} are relevant for computing $\varepsilon$-factors when some of the fields $ k( \tau_1 ) $, $ k( \tau_2 ) $, and $ k( \tau_3 ) $ coincide.
    
\begin{prpd}[{\cite[Theorem 1.2]{Tunnell}}]
    \label{t 1.2}
    Let $F$ be a local non-archimedean field and let $\zeta$ and $\xi$ be characters of $F^\times$ with $ \mathrm{cond} ( \xi ) \ge2 \mathrm{cond} ( \zeta ) $. Let $ y \in F $ be defined by the relation $ \xi(1 + x) = \psi(yx) $ for all $ x \in F $ with $ v(x) \ge \mathrm{cond} ( \xi ) / 2 $. Then 
    \[
        \varepsilon ( \zeta \xi, \psi ) = \zeta^{-1} (y) \varepsilon( \xi, \psi).
    \]
    If $ \mathrm{cond} ( \xi ) = 0 $, we take $ y = \varpi^{ -\mathrm{cond} ( \psi ) } $. Here $\varpi$ is a uniformizing parameter of $F$ and $ \mathrm{cond} ( \cdot ) $ denotes the conductor of a character or an additive character.
\end{prpd}
    
The following property is a result of Frolich-Queryut-Deligne.

\begin{prpd}[{\cite[Theorem 1.4]{Tunnell}}]
    \label{t 1.4}
    Let $F$ be a local non-archimedean field and $E$ be a quadratic separable extension of $F$. Let $\chi$ be a character of $E^\times$ which is trivial on $F^\times$. Let $\alpha$ be an element of $E^\times$ with $ \tr^E_F ( \alpha ) = 0 $. Then $ \varepsilon \left( \chi, \psi_E \right) = c \chi ( \alpha ) $, where $c$ is a constant independent of $\chi$ and $\psi$ is some nontrivial additive character of $F$.
\end{prpd}

Now suppose that representations $\tau_1$ and $\tau_2$ correspond to $ \mathrm{Ind}^k_{ F_1 } \xi_1 $ and $ \mathrm{Ind}^k_{ F_1 } \xi_1' $ respectively, while $\tau_3$ corresponds to $ \mathrm{Ind}^k_{ F_2 } \xi_2 $, where $F_1$ and $F_2$ are distinct quadratic extensions of $k$, and $\xi_1$, $\xi_1'$, $\xi_2$ are characters of $F_1^\times$, $F_1^\times$ and, $F_2^\times$ respectively. Define $\xi$ and $ \widetilde{ \xi } $ by
\[
    \xi = \left( \left( \xi_1 \xi_1' \right) \circ \mathrm{N}^L_{ F_1 } \right) \left( \xi_2 \circ \mathrm{N}^L_{ F_2 } \right) \left( \eta^{-1} \circ \mathrm{N}^L_k \right)
\]
and
\[
    \widetilde{ \xi } = \left( \left( \xi_1^{ s_1 } \xi_1' \right) \circ \mathrm{N}^L_{ F_1 } \right) \left( \xi_2 \circ \mathrm{N}^L_{ F_2 } \right) \left( \eta^{-1} \circ \mathrm{N}^L_k \right),
\]
where $ s_1 \in \Gal ( F_1/k ) $ denotes the nontrivial element.  Let $F_3$ be the third quadratic extension of $k$ and let $\Sigma_{ L/F_3 }$ be a character of $L^\times$ which extends the character $ \varsigma_{ L/F_3 } $ of $F_3^\times$. 
        
Fix a nontrivial additive character $\psi$ of $k$. Define $ y \in L $ to be the element such that 
\[
    \left( \xi \Sigma_{ L/F_3 } \right) (1 + x) = \psi_L ( yx )
\]
for all $x$ with $ v(x) \ge \mathrm{cond} \left( \xi \Sigma_{ L/F_3 } \right) / 2 $. If $ \mathrm{cond} \left( \xi \Sigma_{ L/F_3 } \right) = 0 $, we take $ y = \varpi^{ -\mathrm{cond} ( \psi ) } $ for some uniformizing parameter $\varpi$ of $k$. Similarly, let $ \widetilde{y} $ be defined by replacing $\xi$ with $ \widetilde{ \xi } $ in the definition of $y$.

\begin{prpd}
    \label{cr2}
    Let $\tau_1, \tau_2$, and $\tau_3$ be irreducible  supercuspidal representations of $\GL_2(k)$ with central characters satisfying $\omega_{ \tau_1 } \omega_{ \tau_2 } \omega_{ \tau_3 } = \eta^2 $. Assume the following conductor conditions hold:
    \[
        \mathrm{cond} \left( \xi \Sigma_{ L/F_3 } \right), \mathrm{cond} \left( \widetilde{ \xi } \Sigma_{ L/F_3 } \right) \ge 2 \mathrm{cond} \left( \Sigma_{ L/F_3 } \right).
    \]
    Then
    \[
        \dim \Hom_{ H_0 } \left( \tau_1 \otimes \tau_2 \otimes \tau_3, \eta \circ \det \right) = 1
    \]
    if and only if 
    \[
        \Sigma_{ L/F_3 } ( y^{-1} \widetilde{y} ) = 1.
    \]
\end{prpd}
    
\begin{proof}
    We fix a nontrivial additive character $\psi$ of $k$. Recall that $\varepsilon_0$ denotes the epsilon factor
    \[
        \varepsilon \left( \mathrm{Ind}^k_{ F_1 }  \xi_1 \otimes \mathrm{Ind}^k_{ F_1 } \xi_1' \otimes \mathrm{Ind}^k_{ F_2 } \xi_2 \left( \eta^{-1} \circ \mathrm{N}^{ F_2 }_k \right) \right).
    \]
    As in the previous proposition, Mackey's theorem gives
    \[
        \mathrm{Ind}^k_{ F_1 } \xi_1 \otimes \mathrm{Ind}^k_{ F_1 } \xi_1' \cong \mathrm{Ind}^k_{ F_1 } \xi_1 \xi_1' \oplus \mathrm{Ind}^k_{ F_1 } \xi_1^{ s_1 } \xi_1',
    \]
    which implies that
    \[
        \varepsilon_0 = \varepsilon \left( \mathrm{Ind}^k_{ F_1 } \xi_1 \xi_1' \otimes \mathrm{Ind}^k_{ F_2 } \xi_2 \left( \eta^{-1} \circ \N^{F_2}_k \right), \psi \right) \cdot \varepsilon \left( \mathrm{Ind}^k_{ F_1 } \xi_1^{ s_1 } \xi_1' \otimes \mathrm{Ind}^k_{ F_2 } \xi_2 \left( \eta^{-1} \circ \N^{F_2}_k \right), \psi \right).
    \]
    Let $\rho$ denote the representation $ \mathrm{Ind}^k_{ F_2 } \xi_2 \left( \eta^{-1} \circ \mathrm{N}^{ F_2 }_k \right) $. Another application of Mackey's theorem shows that
    \[
        \rho \otimes \mathrm{Ind}^k_{ F_1 } \xi_1 \xi_1' \cong \mathrm{Ind}^k_{ F_1 } \left( \rho|_{ F_1 } \otimes \xi_1 \xi_1' \right)
    \]
    and
    \[
        \rho \otimes \mathrm{Ind}^k_{ F_1 } \xi_1^{ s_1 } \xi_1' \cong \mathrm{Ind}^k_{ F_1 } \left( \rho|_{ F_1 } \otimes \xi_1^{ s_1 } \xi_1' \right),
    \]
    where $ \rho|_{ F_1 } $ stands for the restriction $ \Res^{ W_k }_{ W_{ F_1 } } \rho $. Therefore,
    \[
        \begin{aligned}
            \varepsilon_0 =& \varepsilon \left( \rho|_{ F_1 } \otimes \xi_1 \xi_1', \psi_{ F_1 } \right) \varepsilon \left( \rho|_{ F_1 } \otimes \xi_1^{ s_1 } \xi_1', \psi_{ F_1 } \right) \varsigma_{ F_1/k }^2 (-1)        \\
                          =& \varepsilon \left( \rho|_{ F_1 } \otimes \xi_1 \xi_1', \psi_{ F_1 } \right) \varepsilon \left( \rho|_{ F_1 } \otimes \xi_1^{ s_1 }\xi_1', \psi_{ F_1 } \right).
        \end{aligned}
    \]
    The subsequent proof consists of two parts: first, we  show that
    \[
        \varepsilon \left( \rho|_{ F_1 } \otimes \xi_1 \xi_1', \psi_{ F_1 } \right)^2 = \varepsilon \left( \rho|_{ F_1 } \otimes \xi_1^{ s_1 } \xi_1', \psi_{ F_1 } \right)^2 = 1,
    \]
    and then we explicitly compute the $ \varepsilon_0 $.

    The proof of the first part adapts the method used in the proof of \cite[Lemma 1.5]{Tunnell}. Recall that the product of the central characters of $\tau_i$ for $ i = 1, 2, 3 $ equals $ \eta^2 $. We have
    \begin{equation}
        \label{centralprod}
        \eta^2 = \omega_{ \tau_1 } \omega_{ \tau_2 } \omega_{ \tau_3 } = \left( \xi_1 \xi_1' \xi_2 \right)|_{ k^\times } \varsigma^2_{ F_1/k } \varsigma_{ F_2/k } = \left( \xi_1 \xi_1' \xi_2 \right)|_{ k^\times } \varsigma_{ F_2/k }.
    \end{equation}       
    The contragredient of $ \rho|_{ F_1 } \otimes \xi_1 \xi_1' $ is
    \[
        \begin{aligned}
            \left( \rho|_{ F_1 } \otimes \xi_1 \xi_1' \right)^\vee
            \cong& \det \left( \rho|_{ F_1 } \otimes \xi_1 \xi_1' \right)^{-1} \otimes \left( \rho|_{ F_1 } \otimes \xi_1 \xi_1' \right)\\
                =& \left( \left( \xi_2|_{ k^\times } \varsigma_{ F_2/k } \right)^{-1} \circ \N^{ F_2 }_k \right) \left( \eta^{-1} \circ \N^{L}_k \right) ( \xi_1 \xi_1' )^{-1} \otimes \rho|_{ F_1 }.
        \end{aligned}
    \]
    Let $ \phi_2 \in \Gal( L/k ) $ fix $ F_2 $. Consider the representation $ \left( \rho|_{ F_1 } \otimes \xi_1 \xi_1' \right)^{ \phi_2 } $ twisted by $\phi_2$. Since $ \rho \cong \rho^{ \phi_2 } $ is $\phi_2$-invariant, we have $ \left( \rho|_{ F_1 } \otimes \xi_1 \xi_1' \right)^{ \phi_2 } \cong \rho|_{ F_1 } \otimes ( \xi_1 \xi_1' )^{ \phi_2 } $. Let $ x \in F_1 $. Then $ x \cdot \phi_2(x) = x \cdot s_1(x) = \N^{ F_1 }_k (x) \in k $. Together with the equation $ (\ref{centralprod}) $, this implies
    \[
        \begin{aligned}
            \left( \eta \circ \N^L_k \right) (x) =& \eta^2 ( x \cdot \phi_2 x ) = \left( \left( \xi_1 \xi_1' \xi_2 \right)|_{ k^\times } \varsigma_{ F_2/k } \right) ( x \cdot \phi_2 x )\\
                                                 =& \left( \xi_2|_{ k^\times } \varsigma_{ F_2/k } \right) \left( \N^{ F_2 }_k (x) \right) \cdot \left( \xi_1 \xi_1' \xi_1^{ \phi_2 } \xi_1'^{ \phi_2 } \right) (x).
        \end{aligned}
    \]
    Hence, as characters of $F_1^\times$,
    \[
        \left( \left( \xi_2|_{ k^\times } \varsigma_{ F_2/k } \right) \circ \N^{ F_2 }_k \right) \left( \eta^{-1} \circ \N^{L}_k \right) ( \xi_1 \xi_1' )^{-1} = \xi_1^{ \phi_2 } \xi_1'^{ \phi_2 }.
    \]
    This implies
    \begin{equation}
        \label{rhocongrho}
        \left( \rho|_{ F_1 } \otimes \xi_1 \xi_1' \right)^\vee \cong \left( \rho|_{ F_1 } \otimes \xi_1 \xi_1' \right)^{ \phi_2 }.
    \end{equation}       
    We now compute:
    \[
        \begin{aligned}
             & \varepsilon \left( \rho|_{ F_1 } \otimes \xi_1 \xi_1', \psi_{ F_1 } \right)^2\\
            =& \varepsilon \left( \rho|_{ F_1 } \otimes \xi_1 \xi_1', \psi_{ F_1 } \right) \varepsilon \left( \left( \rho|_{ F_1 } \otimes \xi_1 \xi_1' \right)^{ \phi_2 }, \psi_{ F_1 } \right)                                      & ( \text{by } ( \ref{twistepsilon} ) )                                                          \\
            =& \varepsilon \left( \rho|_{ F_1 } \otimes \xi_1 \xi_1', \psi_{ F_1 } \right) \varepsilon \left( \left( \rho|_{ F_1 } \otimes \xi_1 \xi_1' \right)^\vee, \psi_{ F_1 } \right)                                             & ( \text{by } ( \ref{rhocongrho} ) )                                                            \\
            =& \det \left( \rho|_{ F_1 } \otimes \xi_1 \xi_1' \right) & ( \text{by } ( \ref{contraepsilon} ) )                                                         \\
            =& \left( \left( \xi_1 \xi_1' \xi_2 \right)|_{ k^\times } \eta^{-2} \varsigma_{ F_2/k } \right) (-1)                                                            \\
            =& 1.                                                     & ( \text{by }( \ref{centralprod} ) )
        \end{aligned}
    \]
    By the same reasoning, $ \varepsilon \left( \rho|_{ F_1 } \otimes \xi_1^{ s_1 } \xi_1', \psi_{ F_1 } \right)^2 = 1 $.
    
    Recall that $\xi$ and $\widetilde{\xi}$ denote the characters of $L^\times$
    \[
        \left( \left( \xi_1 \xi_1' \right) \circ \mathrm{N}^L_{ F_1 } \right) \left( \xi_2 \circ \mathrm{N}^L_{ F_2 } \right) \left( \eta^{-1} \circ \mathrm{N}^L_k \right) \text{ and } \left( \left( \xi_1^{ s_1 } \xi_1' \right) \circ \mathrm{N}^L_{ F_1 } \right) \left( \xi_2 \circ \mathrm{N}^L_{F_2} \right) \left( \eta^{-1} \circ \mathrm{N}^L_k \right),
    \]
    respectively. By Mackey's theorem, we obtain the following isomorphisms:
    \[
        \begin{aligned}
            \rho|_{ F_1 } \otimes \xi_1 \xi_1'
            \cong& \mathrm{Ind}^{ F_1 }_{L} \left( \xi_2 \circ \mathrm{N}^L_{ F_2 } \right) \left( \eta^{-1} \circ \mathrm{N}^L_{k} \right) \otimes \xi_1 \xi_1'\\
            \cong& \Ind^{ F_1 }_L \left( \left( \xi_1 \xi_1' \right) \circ \mathrm{N}^L_{ F_1 } \right) \left( \xi_2 \circ \mathrm{N}^L_{ F_2 } \right) \left( \eta^{-1} \circ \mathrm{N}^L_{k} \right) \cong \Ind^{ F_1 }_{L} \xi.
        \end{aligned}
    \]
    It follows that
    \[
        \varepsilon \left( \rho|_{ F_1 } \otimes \xi_1 \xi_1', \psi_{ F_1 } \right) = \varepsilon \left( \Ind^{ F_1 }_{L} \xi, \psi_{ F_1 } \right) = \varepsilon \left( \xi, \psi_L \right) \lambda_{ L/F_1 } \left( \psi_{ F_1 } \right).
    \]
    A parallel computation shows that an analogous identity holds for $ \rho|_{ F_1 } \otimes \xi_1^{ s_1 } \xi_1' $ and $ \widetilde{ \xi } $. Let $ \phi_i \in \Gal( L/F_i ) $ be the automorphism fixing $F_i$ for each $ i = 1, 2 $ and define $ \phi_3 := \phi_1 \phi_2 \in \Gal( L/k ) $. Since $F_3$ is the subfield of $L$ fixed by $\phi_3$, for each $ x \in F_3 $, we have $ x \cdot \phi_1 x = x \cdot \phi_2 x \in k $. Consequently,
    \[
        \begin{aligned}
            \xi (x) =& \left( \xi_1 \xi_1' \right) (x \cdot \phi_1 x ) \cdot \xi_2 ( x \cdot \phi_2 x ) \cdot \eta^{-1} ( x \cdot \phi_1 x \cdot \phi_2 x \cdot \phi_3 x )\\
                    =& ( \xi_1 \xi_1' \xi_2 \eta^{-2} ) ( x \cdot \phi_1 x ) = \varsigma_{ F_2/k }^{-1} ( x \cdot \phi_1 x ). & (\text{by } ( \ref{centralprod} ) )
        \end{aligned}
    \]
    Hence, the character $ \xi|_{ F_3^\times } \left( \varsigma_{ F_2/L } \circ \mathrm{N}^{L}_{ F_3 } \right) $ is trivial. Now, applying Mackey's theorem yields   
    \[
        \mathrm{Res}_{ F_3 } \mathrm{Ind}^k_{ F_2 } 1_{ F_2 } \cong \mathrm{Ind}_{L}^{ F_3 } 1_L,
    \]
    which implies
    \[
        \varsigma_{ F_2/k } \circ \mathrm{N}^{ F_3 }_k = \det \mathrm{Ind}_{L}^{ F_3 } 1_L = \varsigma_{ L/F_3 }.
    \]
    Choose $ \Sigma_{ L/F_3 } $ to be a character of $L^\times$ extending $ \varsigma_{ L/F_3 } $. Then the character $ \xi \Sigma_{ L/F_3 } $ of $L^\times$ is trivial on $F_3^\times$. By Proposition \ref{t 1.4},
    \[
        \varepsilon \left( \xi \Sigma_{ L/F_3 }, \psi_L \right) = c_{ \alpha } \xi( \alpha ) \Sigma_{ L/F_3 } ( \alpha ),
    \]
    where $ \alpha \in L^\times $ satisfies $ \tr^L_{ F_3 } ( \alpha ) = 0 $, and $c_{\alpha}$ is a constant independent of $\xi$ and $ \Sigma_{ L/F_3 } $. Note that $ \mathrm{cond} \left( \xi \Sigma_{ L/{F_3} } \right) \ge 2 \mathrm{cond} \left( \Sigma_{ L/F_3 } \right) $. By Proposition \ref{t 1.2}, we have
    \[
        \varepsilon \left( \xi, \psi_L \right) = \Sigma_{ L/F_3 }(y) \varepsilon \left( \xi \Sigma_{ L/F_3 }, \psi_L \right) = c_{\alpha} \xi ( \alpha ) \Sigma_{ L/F_3 } ( y \alpha ),
    \]
    where $ y\in L $ satisfies $ \xi \Sigma_{ L/F_3 }(1 + x) = \psi_L( yx )$ for all $ x \in L $ satisfying that $ v(x) \ge \mathrm{cond} \left( \xi \Sigma_{ L/F_3 } \right) / 2 $. Similarly,
    \[
        \varepsilon \left( \widetilde{ \xi }, \psi_L \right) = \Sigma_{ L/F_3 } \left( \widetilde{y} \right) \varepsilon \left( \widetilde{ \xi } \Sigma_{ L/F_3 }, \psi_L \right) = c_\alpha \widetilde{ \xi }( \alpha ) \Sigma_{ L/F_3 } \left( \widetilde{y} \alpha \right),
    \]
    where $ \widetilde{y} $ is defined analogously for $ \widetilde{ \xi } $. We thus obtain that
    \[
        \begin{aligned}
            \varepsilon_0 &= \frac{ \varepsilon_0 }{ \varepsilon \left( \rho|_{ F_1 } \otimes \xi_1^{ s_1 } \xi_1', \psi_{ F_1 } \right)^2 } = \frac{ \varepsilon \left( \rho|_{ F_1 } \otimes \xi_1 \xi_1', \psi_{ F_1 } \right) }{ \varepsilon \left( \rho|_{ F_1 } \otimes \xi_1^{ s_1 } \xi_1', \psi_{ F_1 } \right)} = \frac{ \varepsilon \left( \xi, \psi_L \right) }{ \varepsilon \left( \widetilde{ \xi }, \psi_L \right) }\\
                          &= \frac {c_\alpha \xi ( \alpha ) \Sigma_{ L/F_3 } ( \alpha y )}{ c_{\alpha} \widetilde{ \xi } ( \alpha ) \Sigma_{ L/F_3 } ( \alpha \widetilde{y} ) } = \xi_1 \left( \frac{ \alpha \cdot \phi_1 \alpha }{ \phi_2 \alpha \cdot \phi_3 \alpha } \right) \Sigma_{ L/F_3 }( y^{-1} \widetilde{y} )                                                \\
                          &= \Sigma_{ L/F_3 } ( y^{-1} \widetilde{y} ),
        \end{aligned}
    \]
    where the last equality holds because $\alpha+\phi_3\alpha=\phi_1\alpha+\phi_2\alpha=0$.
\end{proof}

Prasad gave a criterion in the case where the fields $F_1$ and $F_2$ are distinct and $ \mathrm{cond} ( \xi_1 ) \ne \mathrm{cond} ( \xi_1' ) $ (see \cite[Proposition 8.10]{Prasad}).

\begin{prpd}[{\cite[Proposition 8.10]{Prasad}}]
    \label{cr3}
    Let $\tau_1$, $\tau_2$, and $\tau_3$ be irreducible supercuspidal representations of $\GL_2(k)$ with the central characters satisfying $ \omega_{ \tau_1 } \omega_{ \tau_2 } \omega_{ \tau_3 } = \eta^{2} $. Suppose that the representations $\tau_1$ and $\tau_2$ correspond to $ \mathrm{Ind}^k_{ F_1 } \xi_1 $ and $ \mathrm{Ind}^k_{ F_1 } \xi_1' $, respectively while $\tau_3$ corresponds to $ \mathrm{Ind}^k_{ F_2 } \xi_2 $, where $F_1, F_2$ are distinct quadratic extensions of $k$, and $\xi_1$, $\xi_1'$, $\xi_2$ are characters of $F_1^\times$, $F_1^\times$, and $F_2^\times$ respectively. If $\mathrm{cond} ( \xi_1 ) \ne \mathrm{cond} ( \xi_2 ) $, then
    \[
        \dim \Hom_{ H_0 } \left( \tau_1 \otimes \tau_2 \otimes \tau_3, \eta \circ \det \right) = 1.
    \]
\end{prpd}

\subsection{Case $(F,F,F)$}

Suppose that the representation $\tau_i$ corresponds to $ \mathrm{Ind}^k_F \xi_i $ for each $ i = 1, 2, 3 $, where $\xi_i$, for $ i = 1, 2, 3 $ are characters of $F^\times$. Define $\mu_{ij}$ for $ i, j = 0, 1 $ by
\[
    \mu_{ij} = \xi_1^{ s_i } \xi_2^{ s_j } \xi_3 \left( \eta^{-1} \circ { \N^F_k } \right),
\]
where $ \Gal( F/k ) = \{ s_0 ( = 1 ), s_1 \} $. Let $\Sigma_{F/k}$ be a character of $F^\times$ which extends the character $\varsigma_{F/k}$ of $k^\times$.
        
For each $ i, j = 0, 1 $, define $ y_{ij} \in F $ by the relation:
\[
    \left( \mu_{ij} \Sigma_{ F/k } \right) (1 + x) = \psi_F ( y_{ij} x )
\]
for all $ x \in F $ with $ v(x) \ge \mathrm{cond} \left( \mu_{ij} \Sigma_{ F/k } \right) / 2 $. If $ \mathrm{cond} \left( \mu_{ij} \Sigma_{ F/k } \right) = 0 $ for some $ i, j = 0, 1 $, we take $ y_{ij} = \varpi^{ -\mathrm{cond}( \psi ) } $ for a uniformizing parameter $\varpi$ of $k$.

\begin{prpd}
    \label{cr4}
    Let $\tau_1$, $\tau_2$, and $\tau_3$ be irreducible supercuspidal representations of $\GL_2(k)$ with central characters satisfying $ \omega_{ \tau_1 } \omega_{ \tau_2 } \omega_{ \tau_3 } = \eta^2$. Assume the following conductor conditions hold:
    \[
        \mathrm{cond} \left( \mu_{ij} \Sigma_{ F/k } \right) \ge 2 \mathrm{cond} \left( \Sigma_{ F/k } \right), \quad i, j = 0, 1.
    \]
    Then
    \[
        \dim \Hom_{ H_0 } \left( \tau_1 \otimes \tau_2 \otimes \tau_3, \eta \circ \det \right) = 1
    \]
    if and only if
    \[
        \Sigma_{ F/k } \left( \frac{ y_{00} \cdot y_{11} }{ y_{10} \cdot y_{01} } \right) = 1.
    \]
\end{prpd}

\begin{proof}
    The proof is analogous to that of Proposition \ref{cr2}, so we omit the details and state the conclusion directly.
    \begin{itemize}
        \item The condition on central characters $ \omega_{ \tau_1 } \omega_{ \tau_2 } \omega_{ \tau_3 } = \eta^2 $ implies that
        \begin{equation}
            \label{centralprod3}
            ( \xi_1 \xi_2 \xi_3 )|_{ k^\times } \eta^{-2} \varsigma_{ F/k } = 1.
        \end{equation}
        \item Let $\rho$ denote the representation $ \mathrm{Ind}^k_F \xi_3 \eta^{-1} \left( \mathrm{N}^F_k \right) $. Then we have the following isomorphisms:
        \[
            \begin{aligned}
                \mathrm{Ind}^k_F \xi_1 \otimes \mathrm{Ind}^k_F \xi_2 \otimes \rho
                \cong& \left( \mathrm{Ind}^k_F \xi_1 \xi_2 \otimes \rho \right) \oplus \left( \mathrm{Ind}^k_F \xi_1^{ s_1 } \xi_2 \otimes \rho \right)          \\
                \cong& \mathrm{Ind}^k_F \left( \rho|_F \otimes \xi_1 \xi_2 \right) \oplus \mathrm{Ind}^k_F \left( \rho|_F \otimes \xi_1^{ s_1 } \xi_2 \right)\\
                \cong& \bigoplus_{ i, j = 0, 1 } \mathrm{Ind}^k_F \mu_{ij}.
            \end{aligned}
        \]
        \item By $ ( \ref{centralprod3} ) $, we have the isomorphism $ ( \rho|_F \otimes \xi_1 \xi_2 )^\vee \cong ( \rho|_F \otimes \xi_1 \xi_2 )^{ s_1 } $, which implies
        \[
            \varepsilon \left( \rho|_F \otimes \xi_1 \xi_2, \psi_F \right)^2 = \varepsilon \left( \rho|_F \otimes \xi_1^{ s_1 } \xi_2, \psi_F \right)^2 = 1.
        \]
        \item Set $ \varepsilon_0 = \varepsilon \left( \mathrm{Ind}^k_F \xi_1 \otimes \mathrm{Ind}^k_F \xi_2 \otimes \rho \right) $. We have the following equations for $\varepsilon$-factors:
        \[
            \begin{aligned}
                \varepsilon_0 =& \varepsilon \left( \rho|_F \otimes \xi_1 \xi_2, \psi_F \right) \varepsilon \left( \rho|_F \otimes \xi_1^{ s_1 } \xi_2, \psi_F \right)\\
                              =& \frac{ \varepsilon \left( \rho|_F \otimes \xi_1 \xi_2, \psi_F \right) }{ \varepsilon \left( \rho|_F \otimes \xi_1^{ s_1 } \xi_2, \psi_F \right) } = \frac{ \varepsilon \left( \mu_{00}, \psi_F \right) \varepsilon \left( \mu_{11}, \psi_F \right) }{ \varepsilon \left( \mu_{01}, \psi_F \right) \varepsilon \left( \mu_{10}, \psi_F \right) }.
            \end{aligned}
        \]
    \end{itemize}

    We now compute $ \varepsilon( \mu_{ij}, \psi_F ) $ for each $ i, j = 0, 1 $ using Proposition \ref{t 1.2} and Proposition \ref{t 1.4}. Choose $\Sigma_{F/k}$ to be a character of $F^\times$ extending the character $\varsigma_{F/k}$ of $k^\times$. Then by (\ref{centralprod3}), the restriction of the character $\mu_{ij}\Sigma_{F/k}$ to $k^\times$ is trivial for each $ i, j = 0, 1 $:
    \[
        ( \mu_{ij} \Sigma_{F/k} )|_{ k^\times } =( \xi_1 \xi_2 \xi_3 )|_{ k^\times } \varsigma_{ F/k } \eta^{-2} = 1.
    \]
    Applying Proposition \ref{t 1.4} gives that 
    \[
        \varepsilon \left( \mu_{ij} \Sigma_{ F/k }, \psi_F \right) = c_{ \alpha } \mu_{ij}( \alpha ) \Sigma_{ F/k }( \alpha ), \quad i, j = 0, 1,
    \]
    where $ \alpha \in F $ has trace zero over $k$ and $c_\alpha$ is a constant independent of $\mu_{ij}$ and $\Sigma_{F/k}$. By Proposition \ref{t 1.2} and the hypothesis on conductors, we obtain
    \[
        \varepsilon \left( \mu_{ij}, \psi_F \right) = \Sigma_{F/k}( y_{ij} ) \varepsilon \left( \mu_{ij} \Sigma_{ F/k }, \psi_F \right),
    \]
    where for each $ i, j = 0, 1 $, $ y_{ij} \in F $ satisfies
    \[
        \left( \mu_{ij} \Sigma_{F/k} \right) (1 + x) = \psi_F( y_{ij} x )
    \]
    for all $ x \in F $ with $ v(x) \ge \mathrm{cond} \left( \mu_{ij} \Sigma_{ F/k } \right) / 2 $. Therefore,
    \[  
        \varepsilon_0 = \frac{ \varepsilon \left( \mu_{00}, \psi_F \right) \varepsilon \left( \mu_{11}, \psi_F \right) }{ \varepsilon \left( \mu_{01}, \psi_F \right) \varepsilon \left( \mu_{10}, \psi_F \right) } = \frac{ \mu_{00} \mu_{11}( \alpha ) \Sigma_{ F/k }( y_{00} y_{11} ) }{ \mu_{10} \mu_{01} ( \alpha ) \Sigma_{ F/k }( y_{10} y_{01} ) } = \Sigma_{ F/k } \left( \frac{ y_{00} \cdot y_{11} }{ y_{10} \cdot y_{01} } \right).
    \]
\end{proof}

\section*{Acknowledgements}

X. Wang was supported in part by National Key R \& D Program of China (No. 2022YFA1005300). The author sincerely thanks his supervisor, Dongwen Liu, for proposing the problem that guided this study.
The author would like to thank the referee for his/her careful reading and helpful comments, which led to a substantial improvement of the paper.

\end{document}